\documentclass[letterpaper, 11pt,  reqno]{amsart}

\usepackage{amsmath,amssymb,amscd,amsthm,amsxtra, esint}
\usepackage{tikz}

\usepackage{mathtools}
\usepackage{color}
\usepackage[implicit=true]{hyperref}

\usepackage{cases}

\usepackage[left=32mm, right=32mm, 
bottom=27mm]{geometry}

\allowdisplaybreaks[2]

\usepackage{color}

\definecolor{gr}{rgb}   {0.,   0.69,   0.23 }
\definecolor{bl}{rgb}   {0.,   0.5,   1. }
\definecolor{mg}{rgb}   {0.85,  0.,    0.85}
\definecolor{yl}{rgb}   {0.8,  0.7,   0.}
\definecolor{or}{rgb}  {0.7,0.2,0.2}

\usetikzlibrary{shapes.misc}
\usetikzlibrary{shapes.symbols}
\usetikzlibrary{shapes.geometric}

\tikzset{
	ddot/.style={circle,fill=white,draw=black,inner sep=0pt,minimum size=0.8mm},
	>=stealth,
	}

\tikzset{
	ddot2/.style={circle,fill=black,draw=black,inner sep=0pt,minimum size=0.8mm},
	>=stealth,
	}

\newtheorem{theorem}{Theorem} [section]

\newtheorem{lemma}[theorem]{Lemma}
\newtheorem{proposition}[theorem]{Proposition}
\newtheorem{remark}[theorem]{Remark}

\newtheorem{corollary}[theorem]{Corollary}

\DeclareMathOperator*{\supp}{supp}

\newcommand{\1}{\hspace{0.5mm}\text{I}\hspace{0.2mm}}
\newcommand{\II}{\text{I \hspace{-2.8mm} I} }

\newcommand{\I}{\mathcal{I}}

\newcommand{\noi}{\noindent}
\newcommand{\Z}{\mathbb{Z}}
\newcommand{\R}{\mathbb{R}}

\newcommand{\T}{\mathbb{T}}

\let\Re=\undefined\DeclareMathOperator*{\Re}{Re}
\let\Im=\undefined\DeclareMathOperator*{\Im}{Im}

\let\P= \undefined
\newcommand{\P}{\mathbf{P}}

\newcommand{\Q}{\mathbf{Q}}

\newcommand{\al}{\alpha}

\newcommand{\dl}{\delta}

\newcommand{\nb}{\nabla}

\newcommand{\Dl}{\Delta}
\newcommand{\eps}{\varepsilon}

\newcommand{\g}{\gamma}
\newcommand{\G}{\Gamma}
\newcommand{\ld}{\lambda}

\newcommand{\s}{\sigma}
\newcommand{\Si}{\Sigma}
\newcommand{\ft}{\widehat}

\newcommand{\wt}{\widetilde}

\newcommand{\dt}{\partial_t}

\newcommand{\ta}{\theta}

\renewcommand{\l}{\ell}
\renewcommand{\o}{\omega}
\renewcommand{\O}{\Omega}

\newcommand{\les}{\lesssim}

\newcommand{\jb}[1]
{\langle #1 \rangle}

\newcommand{\ind}{\mathbf 1}

\newcommand{\pa}{\partial}

\newcommand{\N}{\mathbb{N}}
\newcommand{\NN}{\mathcal{N}}

\renewcommand{\H}{\mathcal{H}}

\newtheorem*{ackno}{Acknowledgements}

\newcommand{\too}{\longrightarrow}

\newcommand{\Id}{\textup{\bf Id}}

\numberwithin{equation}{section}
\numberwithin{theorem}{section}

\makeatletter
\@namedef{subjclassname@2020}{%
  \textup{2020} Mathematics Subject Classification}
\makeatother

\begin{document}
\baselineskip = 14pt

\title[Probabilistic global well-posedness of vNLW]
{Probabilistic global well-posedness for a viscous nonlinear wave equation modeling 
fluid-structure interaction}

%
\author[J.~Kuan, T.~Oh, and S. \v{C}ani\'{c}]
{Jeffrey Kuan, Tadahiro Oh, and Sun\v{c}ica \v{C}ani\'{c}}

\address{Jeffrey Kuan\\
Department of Mathematics\\
University of California Berkeley\\
Berkeley, CA, USA 94720-3840\\ USA}

\email{jeffreykuan@berkeley.edu}

\address{Tadahiro Oh, School of Mathematics\\
The University of Edinburgh\\
and The Maxwell Institute for the Mathematical Sciences\\
James Clerk Maxwell Building\\
The King's Buildings\\
Peter Guthrie Tait Road\\
Edinburgh\\ 
EH9 3FD\\
 United Kingdom}

\email{hiro.oh@ed.ac.uk}

\address{Sun\v{c}ica \v{C}ani\'{c}\\
Department of Mathematics\\
University of California Berkeley\\
Berkeley, CA, USA 94720-3840\\ USA} 

\email{canics@berkeley.edu}



\begin{abstract}
We prove probabilistic well-posedness for a 2D viscous nonlinear wave equation modeling
fluid-structure interaction between a 3D incompressible, viscous Stokes flow 
and nonlinear elastodynamics of a 2D stretched membrane. The focus is on (rough)  data, often arising in real-life problems,
 for which it is known that the deterministic problem is ill-posed. 
 We show that  random perturbations of such data give rise almost surely to the existence of a unique solution. 
More specifically, we prove almost sure global well-posedness  for a viscous nonlinear wave equation with the subcritical initial data 
in the Sobolev space 
$\H^s (\R^2)$, $s > -  \frac 15$, which are randomly perturbed  using Wiener randomization. 
This result shows ``robustness'' of nonlinear FSI problems/models, and provides confidence that even for  the ``rough data'' (data in $\H^s$, $s >  -\frac 1 5$)
random perturbations of such data 
(due to e.g., randomness in real-life data, numerical discretization, etc.) will almost surely provide a unique solution 
which depends continuously on the data in the $\H^s$ topology.


\end{abstract}

%
\maketitle
%


\section{Background}


This paper is motivated by a study of fluid-structure (FSI) interaction and the impact of {\emph{rough data}} and {\emph{random perturbations}} of such data
on the solution of a nonlinear fluid-structure interaction problem, where the nonlinearity may come, e.g., from a nonlinear forcing of the structure. 
The main motivation derives from real-life applications that often exhibit such data and nonlinear forcing, e.g., coronary arteries contracting and expanding
on the surface of a moving heart, where the forcing comes from the surrounding heart tissue. 

In particular, we are interested in the flow of a viscous incompressible fluid modeled by the Stokes equations
in a channel that is bounded by a two-dimensional membrane, modeled by a nonlinear wave equation. See Fig.~\ref{domain}.
A competition between dissipative effects coming from the fluid viscosity, dispersion and  nonlinear effects coming from the structure model
are of particular interest, especially for the (rough) initial data for which one can show that the Sobolev ${\H^s}$ norm gets inflated 
very quickly after a small time $t_\epsilon > 0$ (the ${\H^s}$ norm gets greater than $1/\epsilon$), 
even for the initial data with a small ${\H^s}$ norm (less than $\epsilon$). For example, due to the nonlinearity in the problem, 
``small'' oscillations in the initial data get amplified quickly, giving rise to an ill-posed problem in Hadamard sense. 
This is often the case for the initial data in ${\H^s}$  with the Sobolev exponent $s$ below a critical exponent $s_\text{crit}$ (rough initial data), 
where $s_\text{crit}$ is ``given'' by the 
``natural'' scaling of the problem. 
In this paper we show that for a range of Sobolev exponents {\emph{below the critical exponent}}, $s_\text{min} < s < s_\text{crit}$ where $s_\text{min} < 0$, random perturbations of such initial data via a Wiener randomization
will still give rise to a globally well-posed problem almost surely. 
This result shows ``robustness'' of nonlinear FSI problems/models, and provides confidence that even for the ``rough data'' (data in $\H^s$, $s\in (s_\text{min}, s_\text{crit})$),
random perturbations due to a combination of factors (e.g., real-life data, numerical discretization, etc.) will almost surely provide a solution 
which depends continuously on the data in the $\H^s$ topology.
\begin{figure}[htp!]
\center
                    \includegraphics[width = \textwidth]{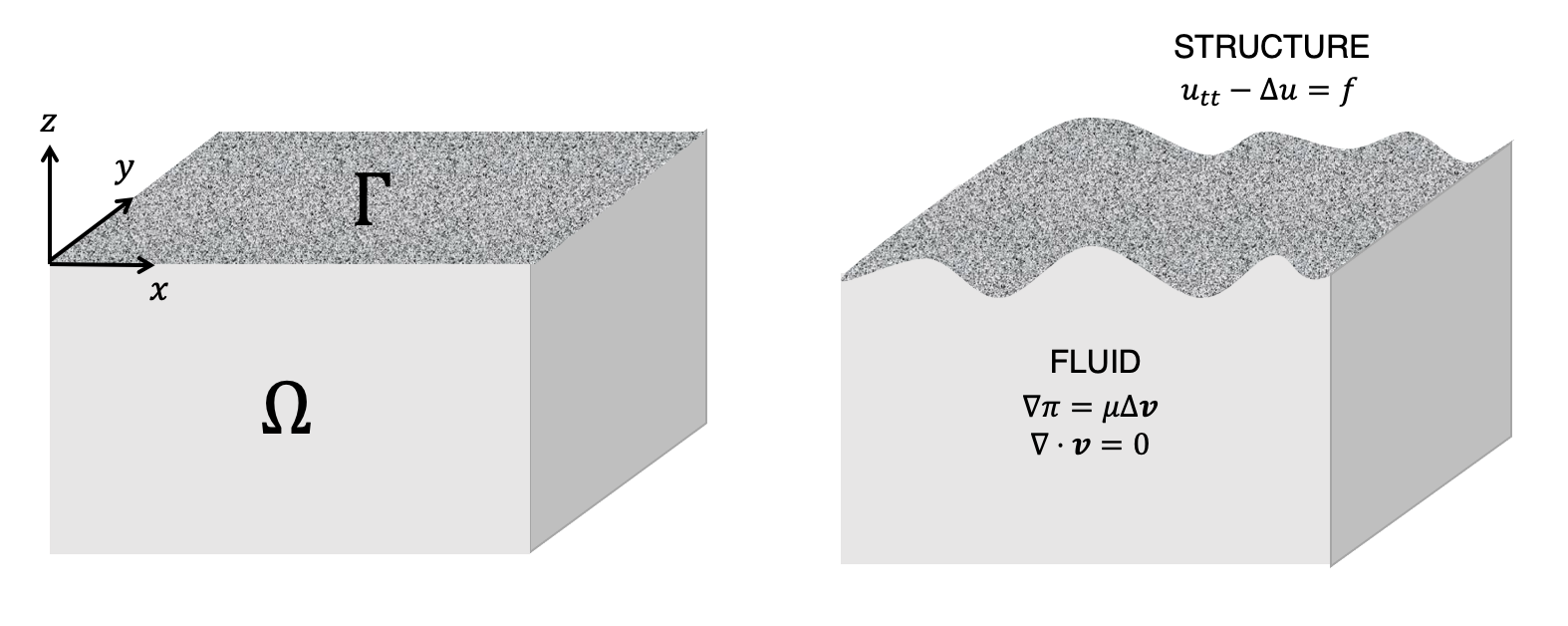}
  \caption{A sketch of the reference configurations for the structure and fluid (left), and the fluid-structure interaction system with nonzero vertical displacement of the structure (right).}
\label{domain}
\end{figure}

More precisely, we study a prototype equation capturing dispersive, dissipative, and nonlinear (forcing) 
effects in fluid-structure interaction problems between the flow of an incompressible, viscous fluid modeled by the 3D Stokes equations, 
and the elastodynamics of a 2D elastic membrane modeled by a nonlinear wave equation. 
The prototype equation is 
the following  {\emph{viscous nonlinear wave equation}} (vNLW) defined on $\R^2$, given by a 2D nonlinear wave equation (NLW)
augmented by the viscoelastic effects modeled by the {\emph{fractional Laplacian operator}} 
(Dirichlet-to-Neumann operator) $D = |\nb| = \sqrt{-\Dl}$ applied to the time derivative $\dt u$, where $u$ denotes vertical membrane displacement:
\begin{align}
\begin{cases}
\dt^2 u - \Dl  u  + 2\mu D \dt u  + |u|^{p-1}u  = 0\\
(u, \dt u)|_{t = 0} = (u_0, u_1)
\end{cases}
\qquad (x, t) \in \R^2 \times \R_+, 
\label{vNLW1}
\end{align}

\noi
Here $\mu > 0$ is a constant denoting fluid viscosity, and $\R_+ = [0, \infty)$.

\if 1 = 0
The viscous nonlinear wave equation arises from a prototypical model for fluid-structure interaction, and models wave dynamics under the influence of viscous regularizing effects. Fluid-structure interaction (FSI) is a physical phenomenon involving the coupled dynamical interaction between a solid and a fluid, where the solid is for instance, deformable with elastic or viscoelastic properties. Such problems feature mathematical difficulties, in terms of the coupling between the solid and fluid equations, and additional geometric nonlinearities that appear in problems in which the fluid domain evolves over time, giving rise to a moving boundary problem. As such, the nonlinear viscous wave equation gives a simplified prototypical model in which one can study well-posedness properties of the FSI system, by isolating the effects of fluid viscosity and nonlinear effects on the elastic structure. 

The mathematical study of FSI is an area of extensive research, motivated by the presence of numerous physical applications. In recent years, many FSI problems motivated by real-life physical systems have been considered in the mathematics literature, including hemodynamics in a curved compliant artery with a coronary stent \cite{CanicCMAME}, the dynamics of floating objects on bodies of water \cite{Lannes}, multilayered poroelastic structures interacting with fluid flow as a model for a bioartificial pancreas \cite{BCMW}, and flow-structure interactions motivated by applications in aeroelasticity \cite{LasieckaAbsorbing, LasieckaSupersonic, LasieckaLongTime, LasieckaRotational, WebsterSubsonic}.

The nature of the two-way coupling between the partial differential equations
(PDEs)
 describing the dynamics of the fluid and the elastodynamics of the structure is crucial in the study of FSI. In mathematical models of FSI, one can consider either linear coupling or nonlinear coupling. The mathematical study of FSI first involved the well-posedness analysis of models with the assumption of \textit{linear coupling}, including a models involving an elastic solid interacting with a fluid described by the linear Stokes equations \cite{Gunzburger}, or the full Navier-Stokes equations \cite{BarGruLasTuff2,BarGruLasTuff,KukavicaTuffahaZiane}. Even though the deformation or displacement of the structure often affects the fluid domain in FSI problems, \textit{linear coupling} is a linearization that assumes the fluid-structure interface is fixed, and evaluates the coupling conditions along the fixed (linearized) fluid-structure interface. Other works have considered the more general case of \textit{nonlinear coupling}, in which the fluid domain is not known a priori. Thus, one must solve a moving boundary problem that has  additional geometric nonlinearities, which complicate the analysis of the problem. See for example, \cite{BdV1,Lequeurre,FSIforBIO_Lukacova,CDEM,CG,MuhaCanic13,
BorSun3d,LengererRuzicka,CSS1,CSS2,Kuk,ChenShkoller,ChengShkollerCoutand,IgnatovaKukavica,Raymod,ignatova2014well,
BorSunSlip,BorSunNonLinearKoiter,BorSunMultiLayered,Grandmont16, CanicCMAME}.

The viscous nonlinear wave equation, first considered in \cite{KC}, arises naturally as a prototypical model for linearly coupled fluid-structure interaction between an elastic structure with nonlinear forcing effects and a fluid that isolates the interaction between the viscous effects of the fluid, nonlinear effects, and the hyperbolic dynamics of the elastic structure. This model features an elastic membrane and a fluid modeled by the stationary Stokes equations, coupled together. Even though this FSI model has two components (the structure and the fluid), the various assumptions in the model, including the linear coupling assumption, give rise to a self-contained equation for the elastodynamics of the structure without reference to the fluid, which is the viscous nonlinear wave equation \eqref{vNLW1}.

The viscous nonlinear wave equation \eqref{vNLW1}, which arises from this fluid-structure interaction model, is given by the usual nonlinear wave equation, considered for example in \cite{TAO}, augmented by the fractional Laplacian operator $D = \sqrt{-\Delta}$ acting on the structure velocity $\partial_{t}u$. The operator $D$ arises naturally as the Dirichlet-Neumann operator for the lower half-plane in $\mathbb{R}^{d}$ (see, for example, \cite{CS} for more information about the fractional Laplacian operator). The presence of this operator in the viscous nonlinear wave equation represents 
the (parabolic) regularizing effects of the fluid viscosity on the structure dynamics. 
For this reason, we study  the viscous nonlinear wave equation \eqref{vNLW1}
by making use of both the dispersive and dissipative properties
of the dynamics.
\fi
%
{\bf{Background on the viscous nonlinear wave equation.}} The viscous nonlinear wave equation \eqref{vNLW1} was derived in \cite{KC} by coupling the elastodynamics of a 2D elastic, prestressed membrane 
whose reference configuration is given by the infinite plane
\begin{equation*}
\Gamma = \{(x, y, 0) \in \mathbb{R}^{3}\},
\end{equation*}
with the flow of an incompressible, viscous Newtonian fluid residing in the lower half space, which we will denote by
\begin{equation*}
\Omega = \{(x, y, z) \in \mathbb{R}^{3} : z < 0\}.
\end{equation*}

\noi
See Figure \ref{domain}.
The membrane and the fluid are \textit{linearly coupled}, namely, the fluid domain remains fixed over time. 
The structure is assumed to only experience displacement in the $z$ direction, which is denoted by $u$,
where $u$ satisfies the following wave equation 
\begin{equation}\label{wave}
\partial_{t}^{2}u - \Delta u = f, \qquad (x, y, t) \in \mathbb{R}^{2} \times \mathbb{R}_{+}.
\end{equation}
Here $f$ is the external loading force on the elastic membrane, which can be nonlinear, as we specify later.

The membrane interacts with the flow of  an incompressible, viscous Newtonian fluid,
defined on the domain $\Omega$, which is fixed in time due to the assumption of linear coupling. 
In order to isolate the dynamical effect of the fluid viscosity on the structure, we model the fluid velocity $\boldsymbol{v} = (v_1, v_2, v_3)$ and pressure $\pi$ by the stationary Stokes equations 
\begin{align}\label{stokes}
\begin{cases}
\nabla \pi = \mu \Delta \boldsymbol{v} \\
\nabla \cdot \boldsymbol{v} = 0
\end{cases} \qquad \text{ on } \Omega,
\end{align}
where the constant, $\mu > 0$, denotes the fluid viscosity. See Figure \ref{domain}.

The fluid and structure are coupled via a two-way coupling, specified by the following two coupling conditions:\begin{enumerate}
\item The \textit{kinematic coupling condition}, which in our problem is a no-slip condition (the trace of the fluid velocity at the interface $\Gamma$ is equal to the 
structure velocity):
\begin{equation}
\boldsymbol{v} = (\partial_{t}u) \boldsymbol{e_{z}}, \qquad \text{ on } \Gamma,
\label{kin1}
\end{equation}

\noi
where $ \boldsymbol{e_{z}} = (0, 0, 1)$; and

\item The \textit{dynamic coupling condition}, which states that the elastodynamics of the membrane is driven by the jump across $\Gamma$ between the normal component 
of the normal Cauchy fluid stress $\boldsymbol{\sigma}$ and the external forcing $F_{\text{ext}}$:
\begin{equation*}
f = -\boldsymbol{\sigma} \boldsymbol{e_{z}} \cdot \boldsymbol{e_{z}}|_{\Gamma} + F_{\text{ext}},
\end{equation*}
where
\begin{equation}\label{kinematic}
\boldsymbol{\sigma} = -\pi \Id  + 2\mu \boldsymbol{D}(\boldsymbol{v}).
\end{equation}
Thus, the structure equation with the dynamic coupling condition reads
\begin{equation}\label{dynamic}
\partial_{t}^{2} u - \Delta u = -\boldsymbol{\sigma} \boldsymbol{e_{z}} \cdot \boldsymbol{e_{z}}|_{\Gamma} + F_{\text{ext}}.
\end{equation}
\end{enumerate}

For completeness, we summarize the derivation here.
We start by noting that the fluid load is given \textit{entirely by the pressure}, due to the particular geometry of this model. Specifically, 
\begin{equation*}
-\boldsymbol{\sigma}\boldsymbol{e_{z}} \cdot \boldsymbol{e_{z}} |_{\Gamma} 
= \bigg( \pi - 2\mu \frac{\partial v_{3}}{\partial z}\bigg)\bigg|_{\Gamma} = \pi|_{\Gamma},
\end{equation*}
which follows from  $\frac{\partial v_{3}}{\partial z}|_\G = 0$ by the incompressibility condition, the kinematic coupling condition \eqref{kin1}, and the fact that $v_{1} = v_{2} = 0$ on $\Gamma$. Hence, we obtain
\begin{equation}\label{pressuredyn}
\partial_{t}^{2} u - \Delta u = \pi|_{\Gamma} + F_{\text{ext}}.
\end{equation}
One can then ``solve'' the stationary Stokes equations \eqref{stokes} with the boundary condition \eqref{kin1} on $\Gamma$ for $\pi$ using a Fourier transform argument, to obtain the final result:
\begin{equation}\label{pressure}
\pi|_{\Gamma} = -2\mu D \dt u.
\end{equation}
We will only sketch the main steps of this derivation 
in the following, and refer  readers to the full derivation in \cite{KC}.

From \eqref{pressuredyn}, we see that the goal is to express  $\pi|_{\Gamma}$ in terms of the structure displacement $u$ and its derivatives. We use the fact that $\pi$ and $\boldsymbol{v}$ satisfy the stationary Stokes equations~\eqref{stokes}
with a boundary condition provided by the kinematic coupling condition~\eqref{kin1}.
We impose a boundary condition at infinity that $\boldsymbol{v}$ is bounded and $\pi$ decreases to zero at infinity. 
From the stationary Stokes equations \eqref{stokes}, one concludes that $\pi$ is a harmonic function in $\Omega$ with a normal derivative along $\Gamma$ given by
\begin{equation}\label{Neumann}
\frac{\partial \pi}{\partial z}\bigg|_{\G} 
= \bigg(\mu \Delta_{x, y} v_{3} + \mu \frac{\partial^{2}v_{3}}{\partial z^{2}}\bigg)\bigg|_{\G} .
\end{equation}

\noi
Hence, 
 we can find $\pi|_{\Gamma}$ in \eqref{pressure} by inverting the Dirichlet-Neumann operator on the 
 lower half space $\Omega$. 
 
 Our main goal is to express
  $v_{3}$ in terms of $u$ and its derivatives, as this will give the Neumann boundary condition for the harmonic function $\pi$ in \eqref{Neumann}. 
  By taking the Laplacian of the first equation in \eqref{Neumann} and recalling that $\pi$ is harmonic, 
  we see that $v_3$ satisfies the biharmonic equation:
\begin{equation}
\Delta^{2} v_{3} = 0
\label{kin3}
\end{equation}
with boundary conditions given by the kinematic boundary condition (see \eqref{kin1}):
\begin{equation}
v_{3}|_{\Gamma} =\dt  u
\label{kin4}
\end{equation}
 Furthermore, there is a boundary condition at infinity that $v_{3}$ must be bounded in 
 $\O$.

By taking the Fourier transform of \eqref{kin3} in the $x$ and $y$ variables 
but not the $z$ variable, we can establish that 
\begin{equation}
\ft v_3 (\xi, z) = \ft {\dt u}(\xi )e^{|\xi|z} - |\xi|\ft{\dt u}(\xi)ze^{|\xi| z}, 
\label{kin5}
\end{equation}

\noi
where $\xi$ denotes the frequency variable corresponding to the $x$ and $y$ variables.
For more details, see the explicit calculation in \cite{KC}. 
Then, by taking the Fourier transform of \eqref{Neumann} in the $x$ and $y$ variables 
and using \eqref{kin4} and \eqref{kin5}, we obtain
\begin{equation}
\frac{\partial \widehat{\pi}}{\partial z}\Big\vert_{\Gamma} = -2\mu|\xi|^{2} \ft {\dt u }(\xi).
\label{kin6}
\end{equation}

\noi
By taking the inverse Fourier transform, this gives 
 the Neumann boundary condition for the harmonic function $\pi$
 in terms of (derivatives) of $u$.
Recall that 
the Dirichlet-Neumann operator for the lower half plane with the vanishing  boundary condition at  infinity is given by $D = \sqrt{-\Delta}$; see \cite{CS}.
By inverting this operator, 
we see that  the Neumann-Dirichlet operator with the same boundary condition at infinity is 
given by the Riesz potential $D^{-1} = (-\Dl)^{-\frac 12}$
with a Fourier multiplier  $|\xi|^{-1}$.
Therefore, by applying the 
 Neumann-Dirichlet operator to (the inverse Fourier transform of) \eqref{kin6}, 
 we obtain the desired result in \eqref{pressure}. 

As a model for nonlinear restoring external forcing effects, we consider a defocusing power nonlinearity of the form
\begin{equation}\label{Fext}
F_{\text{ext}}(u) = -|u|^{p - 1} u,
\end{equation}
for positive integers $p > 1$.
Such a power-type nonlinearity has been studied extensively 
for dispersive equations
 such as the nonlinear Schr\"{o}dinger equations and the nonlinear wave equations; see, for example, \cite{TAO}. Combining \eqref{pressuredyn}, \eqref{pressure}, and \eqref{Fext} gives the final form of the viscous nonlinear wave equation, as stated in \eqref{vNLW1}, in dimension $d = 2$. Although $d = 2$ corresponds 
 to the scenario described in this fluid-structure interaction model, the equation \eqref{vNLW1} can be stated in full generality for arbitrary dimension $d$.

\medskip

{\bf{Critical exponent and ill-posedness.}} Let us now turn to analytical aspects of the viscous NLW \eqref{vNLW1}.
When $\mu \geq 1$, this equation is purely parabolic, 
where the general solution to the homogeneous linear equation
\begin{align*}
\dt^2 u - \Dl  u  + 2\mu D \dt u  = 0
\end{align*}

\noi
with initial data $(u, \partial_{t}u)|_{t = 0} = (u_{0}, u_{1})$, is given by 
\[ u(t) =  e^{-\mu|\nb| + \sqrt{(\mu^2 - 1) |\nb|^2}}
f_1 + e^{-\mu|\nb| - \sqrt{(\mu^2 - 1) |\nb|^2}} f_2.
\]

\noi
Noting that $-\mu |\xi|  + \sqrt{(\mu^2 - 1) |\xi|^2}\sim \mu^{-1} |\xi|$
in this case ($\mu \ge 1$), the solution theory can be studied by 
simply using the Schauder estimate
for the Poisson kernel (see Lemma \ref{LEM:Sch} below).
We will not pursue this direction in this paper.
Instead, 
our main interest in this paper is to study the 
combined effect of the dissipative-dispersive mechanism, 
appearing in \eqref{vNLW1}.
As such, we will restrict our attention to $0 < \mu < 1$.
Without loss of generality, 
we set $\mu = \frac 12$ as in \cite{KC}
and focus on the following version of vNLW:
\begin{align}
\begin{cases}
\dt^2 u - \Dl  u  + D \dt u  + |u|^{p-1} u = 0\\
(u, \dt u) |_{t = 0} = (u_0, u_1).
\end{cases}
\label{vNLW1b}
\end{align}

As in the case of the usual NLW:
\begin{align}
\dt^2 u - \Dl  u   + |u|^{p-1}u  = 0, 
\label{vNLW2}
\end{align}

\noi
the viscous NLW in \eqref{vNLW1b} enjoys the following scaling symmetry.
If $u(x, t)$ is a solution to~\eqref{vNLW1b}, 
then $u^\ld (x, t) = \ld^\frac{2}{p-1} u(\ld x, \ld t)$ is also a solution to \eqref{vNLW1b}
for any $\ld > 0$.
This induces  the critical Sobolev regularity $s_\text{crit}$ on $\R^d$ given by 
\[ s_\text{crit} = \frac d2 - \frac 2{p-1} \] 

\noi
such that the homogeneous Sobolev norm on $\R^2$ remains invariant
under this scaling symmetry.
This scaling heuristics provides a 
common conjecture
 that an evolution equation is well-posed in $H^s$ for $s > s_\text{crit}$, 
 while it is ill-posed for $s < s_\text{crit}$. 
Indeed, for many dispersive PDEs, 
ill-posedness below a scaling critical regularity is known.
In particular, the following form of
strong ill-posedness, known as {\it norm inflation}, 
is established for many dispersive PDEs, 
including NLW; see
\cite{CCT, BT1, CK, Kishimoto,  O1, OW, CP, Ok, Tzvet, OOTz,  FO}.
Norm inflation in the case of the wave equation on $\R^d$ states the following:
given any $\eps > 0$, 
there exist a solution $u$ to \eqref{vNLW2}
and $t_\eps  \in (0, \eps) $ such that 
\begin{align*}
 \| (u, \dt u)(0) \|_{\H^s} < \eps \qquad \text{ but } \qquad \| (u, \dt u)(t_\eps)\|_{\H^s} > \eps^{-1}, 
 \end{align*} 

\noi
where \[\H^s (\R^d) = H^s(\R^d) \times H^{s-1}(\R^d) .\]

In \cite{KC},  Kuan and \v{C}ani\'c
studied this issue
for vNLW \eqref{vNLW1b}.
Due to the presence of the viscous term in \eqref{vNLW1b},
which induces some smoothing property,  
one may  expect to have a different ill-posedness result
but this  was shown not to be the case.
More precisely, 
 Kuan and \v{C}ani\'c proved norm inflation 
for vNLW \eqref{vNLW1b}
in $\H^s(\R^d)$ for $0 < s < s_\text{crit}$
(for any odd integer   $p\geq 3$)
as in the case of the usual NLW.
Moreover, they  showed that the 
viscous contribution has the potential to slow down the speed of the norm inflation.
See \cite{KC} for details.
It is of interest to see if norm inflation 
in negative Sobolev spaces for the usual NLW
\cite{CCT, OOTz, FO}  
carries over to the viscous NLW.
See \cite{dROk}.

This norm inflation for vNLW \eqref{vNLW1b} shows that the equation
is ill-posed in $\H^s(\R^d)$ for $0 < s < s_\text{crit}$, 
showing that there is no hope in 
 studying well-posedness in this low regularity space
in a deterministic manner.

However, we can go beyond the limit of deterministic analysis
and consider our Cauchy problem with randomized initial data.
The area of nonlinear dispersive equations with randomized initial
data has become rather active in recent years \cite{BO96, BT1, CO, LM, BT3, BOP1, BOP2, Poc, OOP, BOP3}.
See also a survey paper \cite{BOP4} in this direction.

In fact, in \cite{KC}
Kuan and \v{C}ani\'c
considered the Cauchy problem \eqref{vNLW1b} 
with $p = 5$ and dimension $d = 2$, and proved almost sure local well-posedness
for {\emph{randomized initial data}} in $\H^s(\R^2)$ for $ s> -\frac 16$.
In this manuscript we extend this result in two important directions:
\begin{enumerate}
\item We prove {\emph{global}} rather than local well-posedness in the probabilistic sense for randomized initial data in $\H^s(\R^2)$ 
$s > s_\text{min}$ where $s_\text{min} = -1/5$; and
\item We extend the interval for the exponents $s$ from $s_\text{min} = -1/6$ to $s_\text{min} = -1/5$,
where the threshold $s_\text{min} = -1/5$ seems to be sharp, see Remark \ref{Remark3.6}(i).
\end{enumerate}
Since the randomized initial data considered in this work are given in terms of Wiener randomization,
we provide a brief description of Wiener randomization next.

{\bf{Wiener randomization.}}
\label{SUBSEC:Wiener}
Let $\psi \in \mathcal{S}(\R^d)$ 
be such that $\supp \psi \subset [-1, 1]^d$,  $\psi(-\xi ) = \overline{\psi(\xi)}$,
and 
\[ \sum_{n \in \Z^d} \psi(\xi - n) \equiv 1 \quad \text{for all }\xi \in \R^d.\]

\noi
Then, any function $f$ on $\R^d$ can be written as
\begin{equation}
f = \sum_{n \in \Z^d} \psi(D-n) f,
\label{B1}
 \end{equation}

\noi
where $ \psi (D-n) $ denotes the Fourier multiplier operator
with symbol $\psi (\,\cdot\, -n)$. Hence, $\psi(D - n)f$ localizes $f$ in the frequency space around the frequency $n \in \mathbb{Z}^{d}$ over a unit scale. 
We recall a particular example of  Bernstein's inequality:
\begin{align}
\|\psi(D-n) f \|_{L^q(\R^d)}
\les \|\psi(D-n) f \|_{L^p(\R^d)}
\label{B2}
\end{align}

\noi
for any $ 1\leq p \leq q \leq \infty$. This classical inequality follows from the localization in the frequency space due to the compact support of $\psi$, and Young's convolution inequality (see, for example, Lemma 2.1 in \cite{LM}).

We now introduce  a randomization 
adapted to the uniform decomposition \eqref{B1}.
For $j = 0, 1$, 
let $\{g_{n,j}\}_{n \in \Z^d}$ be a sequence of mean-zero complex-valued random variables
on a probability space $(\Omega, \mathcal{F}, P)$
such that 
\begin{align}
g_{-n,j}=\overline{g_{n,j}}
\label{K1}
\end{align}
for all $n\in\Z^d$, $j=0,1$.
In particular,  $g_{0, j}$ is real-valued.
Moreover, we  assume that 
$\{g_{0,j}, \Re g_{n,j}, \Im g_{n,j}\}_{n\in\mathcal I, j=0,1}$ are independent,
where the index set $\mathcal{I}$ is defined by 
\begin{equation*}
\mathcal I=\bigcup_{k=0}^{d-1} \Z^k\times \Z_{+}\times \{0\}^{d-k-1}.
\end{equation*}

\noi
Note that $\Z^d = \mathcal I \cup (-\mathcal I)\cup \{0\}$.
Then, given a pair $(u_0, u_1)$ of  functions on $\R^d$, 
we  define the \emph{Wiener randomization} $(u_0^\omega, u_1^\omega)$
of $(u_0,u_1)$ by
\begin{align}
(u_0^\omega, u_1^\omega) 
& = 
\bigg(\sum_{n \in \Z^d} g_{n,0} (\omega) \psi(D-n) u_0,
\sum_{n \in \Z^d} g_{n,1} (\omega) \psi(D-n) u_1\bigg).
\label{R1}
\end{align}

\noi
See \cite{ZF, LM, BOP1, BOP2}.
We emphasize that thanks to \eqref{K1},  this randomization has the desirable property that if $u_0$ and $u_1$ are real-valued,
then their randomizations $u_0^\omega$ and $u_1^\omega$ defined in~\eqref{R1}
are also real-valued.

We make the following assumption on the 
 probability distributions
$\mu_{n,j}$ for $g_{n, j}$;
 there exists $c>0$ such that
\begin{equation}
\int e^{\gamma \cdot x}d\mu_{n,j}(x)\leq e^{c|\gamma|^2}, \quad j = 0, 1, 
\label{B3}
\end{equation}
	
\noindent
for all $n \in \Z^d$,  
(i) all $\gamma \in \R$ when $n = 0$,
and (ii)  all $\g \in \R^2$ when $n \in \Z^d \setminus \{0\}$.
Note that \eqref{B3} is satisfied by
standard complex-valued Gaussian random variables,
standard Bernoulli random variables,
and any random variables with compactly supported distributions.

It is easy to see that, if $(u_0,u_1) \in \H^s(\R^d)$
for some $s \in \R$,
then  the Wiener randomization $(u_0^\omega, u_1^\omega)$ is
almost surely  in $\H^s(\R^d)$.
Note that, under some non-degeneracy condition on the random variables $\{g_{n, j}\}$, 
 there is almost surely no gain from randomization
in terms of differentiability (see, for example, Lemma B.1 in \cite{BT1}).
Instead, the main feature of 
the Wiener randomization 
 \eqref{R1}
is that $(u_0^\omega, u_1^\omega)$  behaves better in terms of integrability.
More precisely, if $u_j \in L^2(\R^d)$, $j=0,1$,
then  the randomized function $u_j^\omega$ is
almost surely  in $L^p(\R^d)$ for any finite $p \geq 2$.
See \cite{BOP1}.

Using the Wiener randomization of the initial data, we prove the main results of this paper,
which are the local and global almost sure existence of a unique solution
for the quintic vNWL in $\R^2$.
More precisely, we have the following main results. 

{\bf{Main results.}}
Fix  $(u_0, u_1) \in \H^s(\R^2)$
for some $s \in \R$ and 
let $(u_0^\o, u_1^\o)$ denote the Wiener randomization
of $(u_0, u_1)$
defined in \eqref{R1}.
Consider the following defocusing quintic vNLW on $\R^2$
with the random initial data:
\begin{align}
\begin{cases}
\dt^2 u - \Dl  u  + D \dt u  + u^5 = 0\\
(u, \dt u) |_{t = 0} = (u_0^\o, u_1^\o)
\end{cases}
\quad (x, t) \in \R^2\times \R_+.
\label{NLW1}
\end{align}


\begin{theorem}\label{THM:LWP1}
Let $s > -\frac 15$.
Then, the quintic vNLW \eqref{NLW1}
is almost surely {\bf{locally}} well-posed
with respect to the Wiener randomization $(u_0^\o, u_1^\o)$
as  initial data.
More precisely, 
there exist $ C, c, \g>0$ and $0 < T_0 \ll1 $ such that
for each $0< T \le T_0$,
there exists a set $\O_T \subset \O$ with the following properties:

\smallskip
\begin{itemize}
\item[\textup{(i)}]
$\displaystyle P(\O_T^c) < C \exp\bigg(-\frac{c}{T^\g \|(u_0, u_1)\|_{\H^s}^2 }\bigg)$,

\medskip

\item[\textup{(ii)}]
For each $\o \in \O_T$, there exists a \textup{(}unique\textup{)} 
local-in-time solution $u$
to \eqref{NLW1}
with $(u, \dt u) |_{t = 0} = (u_0^\o, u_1^\o)$
in the class 
\begin{align*}
V(t)(u_0^\o, u_1^\o) +  C([0,  T]; H^{s_0} (\R^2))
\cap  L^{5+\dl}([0, T];  L^{10}(\R^2))
\end{align*}

\noi
for some $s_0 = s_0(s)> \frac 35$, sufficiently close to $\frac 35$, and small $\dl> 0$ such that $s_0 \geq  1 - \frac 1{5+\dl} - \frac 2{10}$.
Here, $V(t)$ denotes the linear propagator 
for the viscous wave equation defined in \eqref{A1}.

\end{itemize}

\end{theorem}

\begin{remark}\rm 

Let $k_0$ be the smallest integer such that $k_0 \ge T_0^{-1}$.
Then, by setting
\[ \Si = \bigcup_{k = k_0}^\infty \O_{k^{-1}}, \]

\noi
we have:
\begin{enumerate}
\item $P(\Si) = 1$, and 
\item for each $\o\in \Si$, 
there exist a (unique) local-in-time  solution~$u$
to \eqref{NLW1} 
with $(u, \dt u) |_{t = 0} = (u_0^\o, u_1^\o)$
on the time interval $[0, T_\o]$
for some $T_\o > 0$.
More specifically, 
for $\o \in \O_{k^{-1}}$, 
the random local  existence time $T_\o$ is given by $T_\o = k^{-1}$. 
\end{enumerate}
\end{remark}

The proof of 
Theorem \ref{THM:LWP1} is based on 
 the first order expansion \cite{BO96, BT1, CO, BOP1, BOP2, KC}:
\begin{align}
u = z + v, 
\label{exp1}
\end{align}

\noi
where $z = z^\o$ denotes the random {\emph{linear}}  solution given by 
\begin{align}
z(t) = V(t) (u_0^\o, u_1^\o).
\label{z1}
\end{align}

\noi
Then, 
\eqref{NLW1} can be rewritten as
\begin{align}
\begin{cases}
\dt^2 v  -  \Dl  v  + D \dt v  + (v + z)^5 =0\\
(v, \dt v) |_{t = 0} = (0, 0)
\end{cases}
\label{NLW5}
\end{align}

\noi
and we study the fixed point problem \eqref{NLW5} for $v$.
In contrast with \cite{KC}, where the proof of almost sure local well-posedness
was based on the Strichartz estimate 
for the viscous wave equation 
with the diagonal Strichartz space $L^6([0, T]; L^6(\R^2))$,
we prove Theorem \ref{THM:LWP1}, 
using the Schauder estimate for the Poisson kernel (Lemma \ref{LEM:Sch})
and work in a non-diagonal space
$ L^{5+\dl}([0, T];  L^{10}(\R^2))$.
This was important to obtain higher regularity of $\vec{v}$ ($\H^1$ regularity), which allowed us to 
show boundedness of the energy, not otherwise attainable using Strichartz estimates. 
See Section \ref{SEC:LWP}
for details.

\medskip

Once local almost sure well-posedness is established we obtain the following
global almost sure well-posedness result:

\begin{theorem}\label{THM:GWP1}
Let $s > -\frac 15$.
Then, the defocusing quintic vNLW \eqref{NLW1}
is almost surely globally well-posed
with respect to the Wiener randomization $(u_0^\o, u_1^\o)$
as  initial data.
More precisely,
there exists a set $\Sigma \subset \O$
with $P(\Sigma) = 1$
such that,
for each $\o \in \Sigma$, there exists a \text{(}unique\textup{)} global-in-time  solution 
$u$
to \eqref{NLW1}
with $(u, \dt u) |_{t = 0} = (u_0^\o, u_1^\o)$
in the class
\begin{align*}
V(t)(u_0^\o, u_1^\o) +  C(\R_+; H^{s_0} (\R^2))
\end{align*}

\noi
for some $s_0 > \frac 35$.

\end{theorem}

Here, the uniqueness holds in the following sense.
Given any $t_0 \in \R_+$, there exists a random time interval $I(t_0, \o) \ni t_0$
such that the solution $u = u^\o$ constructed in Theorem~\ref{THM:GWP1}
is unique in 
\begin{align*}
V(t)(u_0^\o, u_1^\o) +  C(I(t_0, \o); H^{s_0} (\R^2))
\cap  L^{5+\dl}(I(t_0, \o);  L^{10}(\R^2)), 
\end{align*}

\noi
where $s_0 >\frac 35$ and $\dl > 0$ are as in Theorem \ref{THM:LWP1}.

The main idea of the proof of Theorem \ref{THM:GWP1} is based
on a probabilistic energy estimate, see e.g.,  \cite{BT1, OP}.
With $\vec v = (v, \dt v)$, 
a smooth solution $\vec v $ to the 
defocusing vNLW \eqref{NLW5}
(with $z \equiv 0$ and general initial data)
satisfies monotonicity of the energy (for the usual NLW):
\begin{align}
E(\vec v) = \frac 12 \int_{\R^2} |\nb v|^2 dx + \frac 12 \int_{\R^2} (\dt v)^2 dx
+ \frac 16 \int_{\R^2} v^6 dx .
\label{E0}
\end{align}

\noi
Indeed a simple integration by parts with \eqref{NLW5} (with $z \equiv 0$) shows
\begin{align*}
\dt E(\vec v) = - \|\dt v\|_{\dot H^\frac{1}{2}}^2 \leq 0.
\end{align*}

\noi
For our problem, we proceed with 
 the first order expansion \eqref{exp1}
 and thus 
the residual term $v = u-z$ only satisfies the perturbed vNLW \eqref{NLW5}.
As such,  the monotonicity of the energy $E(\vec v)$ no longer holds.
Nonetheless, by using the time integration by parts trick introduced
by Oh and Pocovnicu  \cite{OP}, 
we establish a Gronwall type estimate 
for $E(\vec v)$ to prove almost sure global well-posedness.

One important point to note is that as it is written, 
the local theory (Theorem \ref{THM:LWP1}) does not provide
a sufficient regularity (i.e.~$\H^1(\R^2)$) for $\vec v$ to guarantee finiteness of the energy $E(\vec v)$.
By using the Schauder estimate (Lemma \ref{LEM:Sch}), however, 
we can show that the residual term $\vec v(t)$ is smoother 
and indeed lies in $\H^1(\R^2)$
for strictly positive times.
It is at this step that the dissipative nature of the equation plays an important role
in this globalization argument.
See Subsection~\ref{SUBSEC:GWP1}
for details.

We conclude this introduction by a few remarks. 

First, we expect that probabilistic {\emph{continuous dependence}},
a notion introduced by Burq and Tzvetkov in \cite{BT3}, see also \cite{Poc},
can be extended from the range $s > -\frac 16$, proved by Kuan and \v{C}ani\'c in \cite{KC},
to the entire range $s > -\frac 15$. We omit details here.


Secondly,  we note that it is also possible to establish almost sure global well-posedness
with respect to the Wiener randomization
for the defocusing vNLW~\eqref{vNLW1b} on $\R^2$
with a {\emph{general defocusing nonlinearity}} $|u|^{p-1} u$ for $p < 5$, 
provided that $s >  - \frac 1p$.
For $p \leq 3$, a straightforward Gronwall type argument
by Burq and Tzvetkov \cite{BT3} applies.
See also \cite{Poc}.
For $3 < p < 5$, one can adapt the argument in Sun and Xia \cite{SX}
which interpolates the $p = 3$ case \cite{BT3}
and the $p = 5$ case \cite{OP} 
in the context of the usual NLW.
See Remark \ref{REM:2}\,(ii)
for a discussion on the $p > 5$ case.

Finally, we remark that the derivation discussed above but with a {\emph{random external forcing}} 
$F_{\text{ext}}$ in \eqref{pressuredyn}, 
leads to a stochastic version of  vNLW.
In \cite{KC2}, Kuan and \v{C}ani\'c
studied the following stochastic vNLW  on $\R^2$ with a multiplicative space-time white noise forcing:
\begin{align}
\dt^2 u - \Dl  u  + 2\mu D \dt u   = F(u)  \xi , 
\label{vNLWx}
\end{align}
where $\xi$ denotes a space-time white noise on $\R^2\times \R_+$.
Under a suitable assumption on $F$, 
they proved global well-posedness of \eqref{vNLWx}.
We also note that well-posedness of stochastic vNLW with a (singular) additive noise
on the two-dimensional torus $\T^2= (\R/\Z)^2$ was recently considered in \cite{LO, Liu}.

The results of this work shed a new light on this active and important research area by showing
the first global well-posedness result for a prototype fluid-structure interaction problem with randomly perturbed rough initial data
in the case when the corresponding deterministic problem is ill-posed.

We begin by presenting the estimates that will be used to obtain the local and global almost sure well-posedness results.

\section{Basic estimates} 

In this section, we go over the deterministic 
and probabilistic linear estimates that will be the basis of the proofs of the main results. 
For this purpose, we introduce the following {\bf{notation}}:

\begin{itemize}
\item
We write $ A \les B $ to denote an estimate of the form $ A \leq CB $ for some $C > 0$.
Similarly, we write  $ A \sim B $ to denote $ A \les B $ and $ B \les A $ and use $ A \ll B $ 
when we have $A \leq c B$ for small $c > 0$.
\item We define the operators $D$ and $\jb{\nb}$ by  setting 
\begin{align}
D = |\nb| = \sqrt{-\Dl} \qquad \text{and}\qquad
\jb{\nb} = \sqrt{ 1- \Dl}, 
\label{nb}
\end{align}
viewed as Fourier multiplier operators
with multipliers $|\xi|$ and $\jb{\xi}$, respectively.
\end{itemize}

\subsection{Linear operators and the relevant linear  estimates}

By writing \eqref{NLW1} in the Duhamel formulation, we have
\begin{align*}
u(t) = V(t) (u_0^\o, u_1^\o) - \int_0^t W(t - t') u^5(t') dt', 
\end{align*}

\noi
where the linear propagator $V(t)$ is defined by 
\begin{align}
V(t) (u_0, u_1)
= e^{- \frac{D}{2}t} \bigg(\cos \big(\tfrac{\sqrt{3}}{2} Dt\big) 
+ \frac{1}{\sqrt{3}}\sin \big(\tfrac{\sqrt{3}}{2} Dt\big) \bigg)u_0
+ e^{- \frac{D}{2}t}\frac{\sin \big(\tfrac{\sqrt{3}}{2} Dt\big)}{\tfrac{\sqrt{3}}{2} D} u_1,
\label{A1}
\end{align}

\noi
and $W(t)$ is defined by 
\begin{align}
W(t) = e^{- \frac{D}{2}t} \frac{\sin \big(\tfrac{\sqrt{3}}{2} Dt\big)}{\tfrac{\sqrt{3}}{2} D}.
\label{A2}
\end{align}

\noi
By letting 
\begin{align}
P(t) = e^{-\frac{D}{2}t}
\label{A2a}
\end{align}

\noi
 denote the Poisson kernel (with a parameter $\frac{t}{2}$)
and 
\begin{align*}
S(t) =  \frac{\sin \big(\tfrac{\sqrt{3}}{2} Dt\big)}{\tfrac{\sqrt{3}}{2} D}, 
\end{align*}

\noi
we have 
\begin{align*}
W(t) = P(t)\circ  S(t) .
\end{align*}

\noi
By defining $U(t)$ by 
\begin{align*}
U (t) (u_0, u_1)
=  \Big(\cos \big(\tfrac{\sqrt{3}}{2} Dt\big) 
+ \frac{1}{\sqrt{3}}\sin \big(\tfrac{\sqrt{3}}{2} Dt\big) \Big)u_0
+ \frac{\sin \big(\tfrac{\sqrt{3}}{2} Dt\big)}{\tfrac{\sqrt{3}}{2} D} u_1, 
\end{align*}

\noi
we have 
\begin{align}
V(t) = P(t) \circ U(t) .
\label{A6}
\end{align}

We first recall 
the Strichartz estimates for the homogeneous linear viscous wave equation
(Theorem 3.2 in \cite{KC}).
Given  $\s > 0$, 
we say that a pair $(q, r)$ is $\s$-admissible 
if $2 \leq q, r \leq \infty$ with $(q, r, \s) \ne (2, \infty, 1)$ 
and 
\begin{align}
\frac 2q + \frac{2\s}{r} \leq \s.
\label{admis1}
\end{align}

\begin{lemma}\label{LEM:Str}
Given $\s > 0$, let $(q, r)$ 
be a $\s$-admissible pair with $r < \infty$.
Then, a solution $u$ to the homogeneous linear wave equation on $\R^d$:
\begin{align*}
\begin{cases}
\dt^2 u - \Dl  u  + D \dt u    = 0\\
(u, \dt u) |_{t = 0} = (u_0, u_1)
\end{cases}
\end{align*}

\noi
satisfies
\begin{align}
\| (u, \dt u) \|_ {L^\infty(\R_+;  \H^s_x(\R^d)) } + 
 \| u  \|_{L^q(\R_+; L^r_x(\R^d))}
\lesssim 
\|(u_0, u_1) \|_{\H^s(\R^d)}, 
\label{hStr}
\end{align}

\noi
provided that the following scaling condition holds:
\begin{align}
 \frac{1}{q} + \frac dr  = \frac d2-  s.
\label{admis2}
\end{align}

\end{lemma}

\begin{remark}\rm
In view of the scaling condition \eqref{admis2}, 
if a pair $(q, r)$ satisfies \eqref{admis2}
for some $s \geq 0$, 
then it is $\s$-admissible with $\s = d$.
\end{remark}

\begin{remark}\rm
We remark that the bounding constant in the estimate \eqref{hStr} depends only on $\sigma > 0$. See \cite{KC} for details.
\end{remark}

\begin{remark}\rm
In the usual Strichartz estimates for the homogeneous wave equation, one must impose an additional restriction on $s$ that $0 \le s \le 1$. This is not present in the corresponding estimate for the homogeneous viscous wave equation in Lemma \ref{LEM:Str}. Although a restriction on $s$ is not explicitly stated in Lemma \ref{LEM:Str}, $s$ does have a limited range of possible values, due to the constraints imposed by the fact that $\sigma > 0$, $2 \le q, r \le \infty$, with $(q, r, \sigma) \ne (2, \infty, 1)$ in \eqref{admis1}, and the scaling condition \eqref{admis2}. The exponent $s$ can take values in the range 
$-\frac 12 < s \le \frac d2$ depending on the choice of parameters. One attains the lower end of the range by taking $q = 2$ and taking $r$ arbitrarily close to $2$, while one attains the upper endpoint of the range by taking $q, r = \infty$ for $s = \frac d2$.
\end{remark}

Next, we state a Schauder-type estimate for the Poisson kernel $P(t)$, 
which allows us to exploit the dissipative nature of
the dynamics.

\begin{lemma}\label{LEM:Sch}
Let $ 1 \leq p \leq q \leq \infty$
and $\al \geq 0$.
Then, we have
\begin{align}
\| D^\al P(t) f\|_{L^q(\R^d)} \les
t^{- \al - d(\frac{1}{p} - \frac{1}{q})}\| f\|_{L^p(\R^d)} 
\label{P1}
\end{align}

\noi
for any $t > 0$.
\end{lemma}

\begin{proof}
Let $K_t(x)$ denote the kernel for $P(t)$, whose Fourier transform is given by 
$\ft K_t(\xi) = e^{-\frac{|\xi|}{2}t}$.
Then, we have 
\begin{align}
 K_t(x) = t^{-d} K_1(t^{-1}x), 
 \label{P2}
\end{align}

\noi
where $K_1(x)$ satisfies
\[ K_1(x) = \frac{c_1}{(c_2 + |x|^2)^{\frac{d + 1}{2}}}\]

\noi
for some $c_1, c_2 > 0$.
In particular, we have $K_1 \in L^r(\R^d)$ for any $1\leq r \leq \infty$.

We first consider the case $\al = 0$.
For $1 \leq r \leq \infty$ with $\frac{1}{r} = \frac{1}{q} - \frac{1}{p} + 1$, 
it follows from \eqref{P2} that 
\begin{align}
\| K_t\|_{L^r} =t^{-d (1 - \frac{1}{r})} 
\| K_1\|_{L^r}
= C_r t^{-d (\frac{1}{p} - \frac{1}{q})}.
\label{P3}
\end{align}

\noi
Then, \eqref{P1} follows from Young's inequality and \eqref{P3}.

Next, we consider the case $\al > 0$.
Noting that 
$D^\al P(t) f = (D^\al K_t) *f$, we need to study the scaling property of $D^\al K_t$.
On the Fourier side, we have
\[ \ft{D^\al K_t}(\xi) = |\xi|^{\al} e^{-\frac{|\xi|}{2}t}
= t^{-\al}  (t |\xi|)^{\al} e^{-\frac{|\xi|t}{2}}
= t^{-\al} \ft {D^\al K_1}(t \xi).
\]

\noi
Namely, we have 
\begin{align}
D^\al K_t(x)  = t^{-d-\al} (D^\al K_1)(t^{-1}x).
\label{P4}
\end{align}

\noi
Then, proceeding as before, 
the bound \eqref{P1} follows from Young's inequality and \eqref{P4}.
\end{proof}

\subsection{Probabilistic estimates}

In this subsection, we establish certain probabilistic Strichartz estimates.
See also Lemma 5.3 in \cite{KC}.

We first  recall the following probabilistic estimate. See \cite{BT1} for the proof.

\begin{lemma} \label{LEM:R1}
Given $j = 0, 1$,
let $\{g_{n, j}\}_{n\in \Z^d}$ be a sequence of mean-zero complex-valued,
 random variables, 
 satisfying  
 \eqref{B3}, as in Subsection \ref{SUBSEC:Wiener}.
   Then, there exists $C>0$ such that
\[ \bigg\| \sum_{n \in \Z^d} g_{n, j}(\omega) c_n\bigg\|_{L^p(\Omega)}
\leq C \sqrt{p} \| c_n\|_{\l^2_n(\Z^d)}\]

\noi
for any $j = 0, 1$, any finite  $p \geq 2$,  and any sequence $\{c_n\} \in \l^2(\Z^d)$.
\end{lemma}

We now establish the first probabilistic Strichartz estimate.

\begin{proposition}\label{PROP:PS}

Given $(u_0, u_1) \in \H^0(\R^d)$, 
let $(u_0^\o, u_1^\o)$ be its Wiener randomization defined in \eqref{R1}, 
satisfying \eqref{B3}.
Then,
given any $2\leq q, r<\infty$ and $\al \geq  0$, 
satisfying $q\al < 1$, 
there exist $C, c>0$ such that
\begin{align}
P\Big(\|D^\al V(t) (u_0^\o,&  u_1^\o) \|_{L^q([0, T]; L^r_x)}> \ld\Big)
\leq C\exp\Bigg(-c\frac{\ld^2}{T^{\frac 2 q -2\al } \|   (u_0, u_1)  \|_{\H^0}^{2}}\Bigg)
\label{PS1}
\end{align}
	
\noi
for any $T > 0$ and $\ld > 0$.

\end{proposition}

\begin{remark}\label{REM:PS1}\rm
(i) From \eqref{PS1}, 
we conclude that 
\[ P\Big( \|D^\al V(t) (u_0^\o,  u_1^\o)\| \le  \lambda \Big) \too 1, \]

\noi
 as $\lambda \to \infty$ for fixed $T > 0$, or as $T \searrow 0$ for fixed $\lambda > 0$. 

\smallskip

\noi
(ii) Let $\al_0 \ge 0$ and $q\al_0 < 1$. Then,  by applying Proposition \ref{PROP:PS}
with $\al = 0$ and $\al = \al_0$,  we have 
\begin{align}
P\Big(\|\jb{\nb}^{\al_0} V(t) (u_0^\o,&  u_1^\o) \|_{L^q([0, T]; L^r_x)}> \ld\Big)
\leq C\exp\Bigg(-c\frac{\ld^2}{T^{\frac 2 q -2\al_0 } \|   (u_0, u_1)  \|_{\H^0}^{2}}\Bigg)
\label{PS1a}
\end{align}
	
\noi
for any $0 < T \le 1$ and $\ld > 0$, 
where $\jb{\nb} = \sqrt {1-\Dl}$ is as in \eqref{nb}.
We also have 
\begin{align}
P\Big(\|\jb{\nb}^{\al_0} V(t) (u_0^\o,&  u_1^\o) \|_{L^q([0, T]; L^r_x)}> \ld\Big)
\leq C\exp\Bigg(-c\frac{\ld^2}{T^{\frac 2 q  } \|   (u_0, u_1)  \|_{\H^0}^{2}}\Bigg)
\label{PS1b}
\end{align}

\noi
for any $ T \ge 1$ and $\ld > 0$.

\end{remark}

See also Lemma 5.3 in \cite{KC}, 
where the case  $q = r = 6$ was treated.
The proof of Proposition \ref{PROP:PS} follows the usual proofs of the probabilistic Strichartz estimates via Minkowski's integral inequality \cite{BT1, CO, BOP1}
but also utilizes 
 the Schauder estimate (Lemma \ref{LEM:Sch}).

\begin{proof}

From \eqref{A6} and 
Lemma \ref{LEM:Sch}
followed by Minkowski's integral inequality, we have 
\begin{align}
\begin{split}
\Big\| \|D^\al V(t) (u_0^\o,&  u_1^\o) \|_{L^q_t([0, T]; L^r_x)} \Big\|_{L^p(\O)}
 \les 
\Big\| \big\| t^{-\al} \|U (t) (u_0^\o, u_1^\o) \|_{L^r_x} \big\|_{L^q_t([0, T])}\Big\|_{L^p(\O)}\\
& \leq 
\Big \| \big\| t^{-\al} \|U (t) (u_0^\o, u_1^\o) \|_{L^p(\O)} \big\|_{L^r_x} \Big\|_{L^q_t([0, T])}
\end{split}
\label{PS2}
\end{align}

\noi
for any finite $p \geq \max (q, r)$.
By 
Lemma \ref{LEM:R1}, 
Minkowski's integral inequality, 
Bernstein's unit-scale inequality \eqref{B2}, 
 and 
the boundedness of $U(t)$ from $\H^0(\R^d)$ into $L^2(\R^d)$,
we obtain
\begin{align}
\begin{split}
\eqref{PS2} & \les \sqrt{p}\, 
\Big \| t^{-\al} \big\|  \|\psi(D-n)  U(t) (u_0, u_1)  \|_{\l^2_n} \big\|_{L^r_x} \Big\|_{L^q_t([0, T])}\\
& \leq \sqrt{p}\, 
\Big \| t^{-\al} \big\|  \|\psi(D-n)  U(t) (u_0, u_1)  \|_{L^r_x} \big\|_{\l^2_n} \Big\|_{L^q_t([0, T])}\\
& \les \sqrt{p}\, 
\Big \| t^{-\al} \|    U(t) (u_0, u_1)  \|_{L^2_x} \Big\|_{L^q_t([0, T])}\\
& \les \sqrt{p} \, T^{\frac 1q - \al} 
 \|   (u_0, u_1)  \|_{\H^0}, 
\end{split}
\label{PS3}
\end{align}

\noi
where we used $q\al < 1$ in the last step.
Then, the tail estimate \eqref{PS1} follows
from \eqref{PS3} and Chebyshev's inequality.
See the proof of Lemma 3 in \cite{BOP1}.\footnote{Lemma 2.2 in the arXiv version.}
\end{proof}

In establishing almost sure global well-posedness, 
we need to introduce several additional linear operators.
  Define  $\wt V(t)  $  by 
\begin{align}
\begin{split}
\wt V(t) (u_0, u_1)
& = \jb{\nb}^{-1} \dt V(t) \\
& =  - \frac{2\sqrt 3}{3}\frac{D}{\jb{\nb}} e^{- \frac{D}{2}t} 
\sin \big(\tfrac{\sqrt{3}}{2} Dt\big) u_0\\
& \quad 
+ 
e^{- \frac{D}{2}t}
\bigg(-\frac 12 \frac{D}{\jb{\nb}}\frac{\sin \big(\tfrac{\sqrt{3}}{2} Dt\big)}{\tfrac{\sqrt{3}}{2} D}
+ \frac{\cos \big(\tfrac{\sqrt{3}}{2} Dt\big)}{\jb{\nb}}
\bigg)
 u_1.
 \end{split}
\label{z3}
\end{align}

\noi
Then, 
\noi
defining $\wt U(t)$ by 
\begin{align*}
\begin{split}
\wt U (t) (u_0, u_1)
& =  - \frac{2\sqrt 3}{3}\frac{D}{\jb{\nb}} 
\sin \big(\tfrac{\sqrt{3}}{2} Dt\big) u_0\\
& \quad 
+ 
\bigg(-\frac 12 \frac{D}{\jb{\nb}}\frac{\sin \big(\tfrac{\sqrt{3}}{2} Dt\big)}{\tfrac{\sqrt{3}}{2} D}
+ \frac{\cos \big(\tfrac{\sqrt{3}}{2} Dt\big)}{\jb{\nb}}
\bigg)
 u_1, 
 \end{split}
\end{align*}

\noi
we have 
\begin{align}
\wt V(t) = P(t) \circ \wt U(t) .
\label{z5}
\end{align}

Next, we state 
 a probabilistic estimate involving the $L^\infty_t$-norm, 
 which plays an important role in establishing an energy bound
 for almost sure global well-posedness.
The proof is based on an adaptation
of the proof of Proposition 3.3 in \cite{OP} combined with the Schauder estimate (Lemma \ref{LEM:Sch}).

\begin{proposition}
\label{PROP:PS2}

Given a pair   $(u_0, u_1)$ of real-valued functions defined on $\R^2$,
let $(u_0^{\omega}, u_1^\omega)$ be its Wiener randomization defined in \eqref{R1}, 
satisfying \eqref{B3}.
Fix $T \gg 1 \geq  T_0 > 0$
and let
$V^*(t) = V(t)$ or $\wt V(t)$ defined in \eqref{A1} and \eqref{z3}, 
respectively.
Then,
given any $2\leq  r\le \infty$, $\al \geq  0$, and $\eps_0 > 0$, 
there exist $C, c>0$ such that
\begin{align*}
P\Big(\|D^\al V^*(t) (u_0^\o,&  u_1^\o) \|_{L^\infty([T_0, T]; L^r_x)}> \ld\Big)
\leq CT \exp\Bigg(-c\frac{\ld^2}{T^2 T_0^{-2\al} \|   (u_0, u_1)  \|_{\H^{\eps_0}}^{2}}\Bigg)
\end{align*}
	
\noi
for any  $\ld > 0$.

\end{proposition}

\begin{proof}
Let $U^*(t) = P(-t) \circ V^*(t)$.
Then, from 
Lemma \ref{LEM:Sch}
with \eqref{A6} or \eqref{z5}, we have 
\begin{align*}
\begin{split}
 \|D^\al V^*(t) (u_0^\o,&  u_1^\o) \|_{L^\infty_t([T_0, T]; L^r_x)} 
 \les 
 \big\| t^{-\al} \|U^* (t) (u_0^\o, u_1^\o) \|_{L^r_x} \big\|_{L^\infty_t([T_0, T])}\\
& \leq 
T_0^{-\al}
 \|U^*(t) (u_0^\o,  u_1^\o) \|_{L^\infty_t([T_0, T]; L^r_x)} 
\end{split}
\end{align*}

\noi
As in the proof of Proposition 3.3 in \cite{OP}, 
the rest follows from Lemma 3.4 in \cite{OP}, 
which established similar $L^\infty_t$-bounds
for the half-wave operators $e^{\pm i tD}$.
\end{proof}

\begin{remark}\rm
It is also possible to prove Proposition \ref{PROP:PS2}, using 
the Garsia-Rodemich-Rumsey inequality (\cite[Theorem A.1]{FV}).
See, for example, Lemma 2.3 in \cite{GKOT} in the context of
the stochastic nonlinear wave equation.
\end{remark}

\section{Local well-posedness}\label{SEC:LWP}

In this section, we present the proof of Theorem \ref{THM:LWP1}.
Instead of \eqref{NLW5} with the zero initial data, we study \eqref{NLW5} with general (deterministic) initial data
$ (v_0, v_1)$:
%
%
\begin{align}
\begin{cases}
\dt^2 v  -  \Dl  v  + D \dt v  + (v + z)^5 =0\\
(v, \dt v) |_{t = 0} = (v_0, v_1).
\end{cases}
\label{NLW6}
\end{align}
We recall from \eqref{z1} and \eqref{A1} that $z = V(t) (u_0^\o, u_1^\o)$ is the random linear solution
with the randomized initial data 
 $(u_{0}^{\omega}, u_{1}^{\omega})$ 
 which is the result of 
 the Wiener randomization  \eqref{R1} performed on the given deterministic initial data $(u_{0}, u_{1}) \in \mathcal{H}^{s}(\mathbb{R}^{2})$.

\begin{theorem}\label{THM:LWP2}
Let $s > -\frac 15$.
Fix $(v_0, v_1) \in \H^{s_0}(\R^2)$
for some $s_0 = s_0(s) > \frac 35$ sufficiently close to~$\frac 35$.
Then, 
there exist $ C, c, \g>0$ and  $0 < T_0 \ll1 $ such that
for each $0< T \le T_0$,
there exists a set $\O_T \subset \O$ with the following properties:

\smallskip
\begin{itemize}
\item[\textup{(i)}] The following probability bound holds:
\begin{equation}\label{probbound1}
\displaystyle P(\O_T^c) < C \exp\bigg(-\frac{c}{T^\g \|(u_0, u_1)\|_{\H^s}^2 }\bigg).
\end{equation}

\medskip

\item[\textup{(ii)}]
For each $\o \in \O_T$, there exists a \textup{(}unique\textup{)} 
solution $(v, \dt v)$ 
to \eqref{NLW6}
with $(v, \dt v) |_{t = 0} = (v_0, v_1)$
in the class 
\begin{align}
 (v, \dt v) \in C([0,  T]; \H^{s_0} (\R^2))
\qquad \text{and}\qquad 
v\in  L^{5+\dl}([0, T];  L^{10}(\R^2))
\label{class1}
\end{align}

\noi
for small $\dl> 0$ such that $s_0 \geq  1 - \frac 1{5+\dl} - \frac 2{10}$.

\end{itemize}

\end{theorem}

In Subsection \ref{SUBSEC:3.1}, 
we first state several linear estimates.
We then present the proof of Theorem \ref{THM:LWP2}
in Subsection \ref{SUBSEC:3.2}.

\subsection{Linear estimates}
\label{SUBSEC:3.1}

In this subsection,  we establish several nonhomogeneous
linear estimates, which are slightly different from 
those in Theorem 3.3 in \cite{KC}.

\begin{lemma}\label{LEM:lin1}
Let $W(t)$ be as in \eqref{A2}.
Then, given sufficiently small $\dl > 0$, we have 
\begin{align}
\bigg\|\int_0^t W(t - t')F(t') dt'\bigg\|_{L^{5+\dl }_t([0, T];  L^{10}_x(\R^2))}
\les \|F\|_{L^1([0, T]; L^2_x(\R^2))}
\label{lin1}
\end{align}

\noi
for any $0 < T \leq 1$.

\end{lemma}

\begin{proof}
Let $\P_{\les 1}$ be a smooth\footnote{Namely, given by 
a smooth Fourier multiplier.} projection onto 
spatial frequencies $\{|\xi|\leq 1\}$
and set $\P_{\gg 1} = \Id - \P_{\les 1}$.
In the following, we separately estimate the contributions
from $\P_{\les 1} F$ and $\P_{\gg 1}F$.

Let us first estimate the low frequency contribution.
By Minkowski's integral inequality and Bernstein's unit-scale inequality \eqref{B2}
with 
$\sin x \leq x$ for $x \geq 0$, we have
\begin{align}
\begin{split}
\bigg\|\int_0^t & W(t - t')   \P_{\les 1}F(t') dt'\bigg\|_{L^{5+\dl}_t([0, T];  L^{10}_x)}\\
&  \leq 
\bigg\|\int_0^t \|\ind_{[0, t]}(t') W(t - t')\P_{\les 1}F(t')\|_{L^{10}_x} 
dt'\bigg\|_{L^{5+\dl}_t([0, T])}\\
& \leq 
\bigg\|\int_0^t (t-t') \|\ind_{[0, t]}(t')\P_{\les 1}F(t')\|_{L^{2}_x} 
dt'\bigg\|_{L^{5+\dl}_t([0, T])}\\
& \les T^\ta \|F\|_{L^1([0, T]; L^2_x)}
\end{split}
\label{lin2}
\end{align}

\noi
for some $\ta > 0$.

Next, we estimate the high frequency contribution. 
Note that the pair $(5+\dl, 10)$ is $\s$-admissible for $\s \geq \frac 12$
in the sense of \eqref{admis1}. 
Let 
\begin{align}
s_0 = 1 - \frac{1}{5+\dl} - \frac{2}{10}
=   \frac{3}{5} +  \dl_0
\label{lin2a}
\end{align}

\noi
for some small $ \dl_0 = \dl_0(\dl)> 0$.
Then, 
by Minkowski's integral inequality
and the homogeneous Strichartz estimate (Lemma \ref{LEM:Str}), we have 
\begin{align}
\begin{split}
\bigg\|\int_0^t&  W(t - t')\P_{\gg1} F(t') dt'\bigg\|_{L^{5+\dl}_t([0, T];  L^{10}_x)}\\
&  \leq 
\int_0^T \|\ind_{[0, t]}(t') W(t - t')\P_{\gg 1}F(t')
\|_{L^{5+\dl}_t([0, T];  L^{10}_x)}
dt'\\
&  \les
\int_0^T \|\P_{\gg1} F(t')
\|_{H^{s_0-1}_x}
dt'\\
& 
\les \|F\|_{L^1([0, T]; L^2_x)}.
\end{split}
\label{lin3}
\end{align}

The desired bound \eqref{lin1} then follows from \eqref{lin2}
and \eqref{lin3}.
\end{proof}
	
\begin{lemma}\label{LEM:lin2}
Let $W(t)$ be as in \eqref{A2}.
Then, given $0 \leq  s \leq 1$, we  have 
\begin{align}
\bigg\|\int_0^t W(t - t')F(t') dt'\bigg\|_{C([0, T];   H^s_x(\R^2))}
& \les \|F\|_{L^1([0, T]; L^2_x(\R^2))}, 
\label{lin4}\\
\bigg\|\dt \int_0^t W(t - t')F(t') dt'\bigg\|_{C([0, T];   H^{s-1}_x(\R^2))}
& \les \|F\|_{L^1([0, T]; L^2_x(\R^2))},
\label{lin5}
\end{align}

\noi
for any $0 < T \leq 1$.

\end{lemma}

\begin{proof}
The first estimate \eqref{lin4} follows from Minkowski's integral inequality with \eqref{A2}.
As for the second estimate \eqref{lin5}, 
we first note from \eqref{A2} that 
\[ \dt \int_0^t W(t - t')F(t') dt'
= \int_0^t \dt W (t - t') F(t') dt', \]

\noi
where 
\[ \dt W(t) 
= e^{- \frac{D}{2}t} \bigg(\cos \big(\tfrac{\sqrt{3}}{2} Dt\big) 
- \frac{1}{\sqrt{3}}\sin \big(\tfrac{\sqrt{3}}{2} Dt\big) \bigg).\]

\noi
Then, the second estimate \eqref{lin5} follows from Minkowski's integral inequality
and the boundedness of $\dt W(t)$ on $H^{s-1}(\R^2)$.
\end{proof}

\subsection{Local well-posedness}
\label{SUBSEC:3.2}

We now present the proof of Theorem \ref{THM:LWP2}.

\begin{proof}[Proof of Theorem \ref{THM:LWP2}]
Fix  $s > -\frac{1}{5}$ and $(u_0, u_1) \in \H^s(\R^2)$.
Then, there exists small $\dl > 0$ such that 
\begin{equation}\label{delta}
s > -\frac{1}{5 + \delta}, 
\end{equation}
and we fix this choice of $\delta > 0$ for the remainder of the proof. 

Fix $C_0 > 0$ and 
define  the event $\O_T = \O_T(C_0)$ by setting 
\begin{equation}
\O_{T} = \big\{ \o \in \O: ||z||_{L^{5+\delta}([0, T]; L^{10}_{x})} \le
C_0 
 \big\}.
 \label{O1}
\end{equation}

\noi
Then, from 
 the probabilistic Strichartz estimate (Proposition \ref{PROP:PS})
 (see also \eqref{PS1a})
 with \eqref{z1}, \eqref{R1}, and  
 \eqref{delta}
  (which guarantees $\al_0 q < 1$ in invoking \eqref{PS1a} with $\al_0 = -s$ and $q = 5+\dl$),  
  we have 
\begin{align}
P(\O_T^c) 
\leq C\exp\Bigg(-c\frac{C_0^2 }{T^{\frac 2 q +2s } \|   (u_0, u_1)  \|_{\H^s}^{2}}\Bigg)
\label{PS1x}
\end{align}
	
\noi
for any $0 < T \le 1$.
 We remark that the choice of $C_0 > 0$ does not matter, 
 and that the specific value of $C_0 > 0$ affects only the size of $T_{0} \ll 1$ and the constants in the estimate~\eqref{probbound1}.

By writing \eqref{NLW6} in the Duhamel formulation, we have
\begin{align*}
v(t) = \G_{(v_0, v_1), z} (v)(t) := V(t) (v_0, v_1) - \int_0^t W(t - t') (v+z)^5(t') dt'.
\end{align*}

\noi
For simplicity, we set $\G =  \G_{(v_0, v_1), z} $. Let $\vec \G(v) = (\G(v), \dt \G(v))$.
Let $s_0 = s_0(\dl) = \frac 35+\dl_0$ as in \eqref{lin2a}.
Then, given $T > 0$, 
define the solution space $Z(T)$ by setting
\begin{align*}
Z(T) = X(T) \times Y(T), 
\end{align*}

\noi
where $X(T)$ and $Y(T)$ are defined by 
\begin{align*}
X(T) &  = C([0, T]; H^{s_0}(\R^2))\cap 
L^{5+\dl}([0, T];  L^{10}(\R^2))\\
Y(T) & =  C([0, T]; H^{s_0-1}(\R^2)).
\end{align*}

\noi
In order to prove 
Theorem \ref{THM:LWP2}, 
we  show that 
there exists small $ 0 < T_{0} \ll 1$ such that  $\vec \G: (v, \partial_{t}v) \mapsto (\G(v), \dt \G(v))$ 
is a strict contraction  on an appropriate closed ball in $Z(T)$ for any $0 < T \le T_{0}$
and for any $\o \in \O_T$, where $\O_T$
is as in \eqref{O1}.
The probability estimate~\eqref{probbound1} on $\O_T^c$ follows
from \eqref{PS1x}.

Fix  arbitrary $\omega \in \O_T$ for $0 < T \le T_{0}$, where $T_0$ is to be determined later.
 Recall $\vec \G(v) = (\G(v), \dt \G(v))$.
Note that the ordered pair $(5 + \delta, 10)$ is $\sigma$-admissible for $\sigma \ge \frac{1}{2}$ in the sense of \eqref{admis1} and furthermore, it satisfies the scaling condition \eqref{admis2} with $s_{0}$ as defined in \eqref{lin2a}. Then, by Lemmas \ref{LEM:Str}, \ref{LEM:lin1}, and \ref{LEM:lin2}
with \eqref{O1}, 
we have 
\begin{align*}
\|\vec \G(v)\|_{Z(T)}
& \les \| (v_0, v_1) \|_{\H^{s_0}}
+ \|(v+z)^5\|_{L^1([0, T]; L^2_x)}\\
& \les \| (v_0, v_1) \|_{\H^{s_0}}
+ T^\ta \Big(\|v\|_{L^{5+\dl}([0, T]; L^{10}_x)}
+ \|z\|_{L^{5+\dl}([0, T]; L^{10}_x)}\Big)\\
& \les \| (v_0, v_1) \|_{\H^{s_0}}
+ T^\ta \Big(\|\vec v\|_{Z(T)}
+ C_0 \Big)
\end{align*}

\noi
for some $\ta > 0$, where $\vec v = (v, \dt v)$.

A similar computation yields the following difference estimate:
\begin{align*}
\|\vec \G(v) - \vec \G(w)\|_{Z(T)}
& \les 
\|(v+z)^5 - (w+z)^5\|_{L^1([0, T]; L^2_x)}\\
& \les T^\ta \Big( \| v\|_{L^{5+\dl}([0, T]; L^{10}_x)}^4 
+ \| w\|_{L^{5+\dl}([0, T]; L^{10}_x)}^4\\
& \hphantom{XXXX}+ \|z\|_{L^{5+\dl}([0, T]; L^{10}_x)}^4\Big)
 \| v -  w\|_{L^{5+\dl}([0, T]; L^{10}_x)}\\
 & \les T^\ta \Big( \|\vec v\|_{Z(T)}^4 
+ \|\vec w\|_{Z(T)}^4
+ C_0^4\Big)
 \|\vec v - \vec w\|_{Z(T)}.
\end{align*}

\noi
Hence 
by choosing $T_{0} > 0$ sufficiently small, depending on the initial choice of $C_0 > 0$ and $||(v_{0}, v_{1})||_{\mathcal{H}^{s_{0}}}$, we see that $\vec \G
=  \vec \G_{(v_0, v_1), z} $
is a strict contraction on the ball in $Z(T)$ of radius $\sim 
1 + \| (v_0, v_1) \|_{\H^{s_0}}$, 
 whenever $\omega \in \O_T$ and $0 < T \le T_{0}$. 
This proves almost sure local well-posedness of \eqref{NLW6} 
(and \eqref{NLW1}) for $s > -\frac 15$.
This concludes the proof of Theorem~\ref{THM:LWP2}
(and hence of Theorem \ref{THM:LWP1}).
\end{proof}

\medskip

Let us conclude this section by stating some corollaries
and remarks.
Given $N \in \N$, let $\P_{\le N}$ denote
a smooth projection onto the (spatial) frequencies $\{|\xi|\leq  N\}$.
Then, consider the following perturbed vNLW:
\begin{align}
\begin{cases}
\dt^2 v_N  -  \Dl  v_N  + D \dt v_N  + (v_N + z_N)^5 =0\\
(v_N, \dt v_N) |_{t = 0} = (\P_{\le N}v_0, \P_{\le N}v_1), 
\end{cases}
\label{NLW6b}
\end{align}

\noi
where $z_N$ denotes
the truncated random linear solution defined by 
\begin{align*}
z_N(t) = V(t) (\P_{\le N}u_0^\o, \P_{\le N}u_1^\o).
\end{align*}

\noi
Then, a slight modification of the proof of Theorem \ref{THM:LWP2} yields
the following approximation result.

\begin{corollary}\label{COR:LWP3}
Let $s > -\frac 15$
and $s_0 > \frac 35$ be as in Theorem \ref{THM:LWP2}.
Fix $(v_0, v_1) \in \H^{s_0}(\R^2)$.
Let $\O_T$ be as in Theorem~\ref{THM:LWP2}.
Furthermore, for each $\o \in \O_T$, let  
 $(v, \dt v)$ 
be the solution to \eqref{NLW6} on $[0, T]$
with $(v, \dt v) |_{t = 0} = (v_0, v_1)$
constructed in Theorem \ref{THM:LWP2}.
By possibly shrinking the local existence time $T$
by a constant factor \textup{(}while keeping 
the definition \eqref{O1} of $\O_T$ unchanged\textup{)}, 
for each $\o \in \O_T$, 
the solution  
 $(v_N, \dt v_N)$ 
 to \eqref{NLW6b}
converges to $(v, \dt v)$ in the class \eqref{class1}
as $N \to \infty$.

\end{corollary}

Next, consider the 
following perturbed vNLW:
\begin{align}
\begin{cases}
\dt^2 v  -  \Dl  v  + D \dt v  + (v + f)^5 =0\\
(v, \dt v) |_{t = t_0} = (v_0, v_1), 
\end{cases}
\label{NLW6a}
\end{align}

\noi
where $f$ is a given deterministic function.
As a corollary to the proof of Theorem \ref{THM:LWP2}, 
we have the following local well-posedness result of \eqref{NLW6a}.

\begin{corollary}\label{COR:LWP4}
Let $s > -\frac 15$,  $s_0 > \frac 35$, 
and small $\dl > 0$  be as in Theorem \ref{THM:LWP2}.
Fix $(v_0, v_1) \in \H^{s_0}(\R^2)$
and fix $t_0\in \R_+$.
Suppose that
\[f \in L^{5+\dl}([t_0, t_0 + 1]; L^{10}(\R^2)).\]
Then, there exists $T = T\big(\| (v_0, v_1) \|_{\H^{s_0}}, 
\|f\|_{L^{5+\dl}([t_0, t_0+ T]; L^{10}_x)}\big) >0$
and 
 a \textup{(}unique\textup{)} 
solution $(v, \dt v)$ 
to \eqref{NLW6a}
on the time interval $[t_0, t_0 + T]$
with $(v, \dt v) |_{t = t_0} = (v_0, v_1)$
in the class
\begin{align*}
 (v, \dt v) \in C([t_0,  t_0 + T]; \H^{s_0} (\R^2))
\qquad \text{and}\qquad 
v\in  L^{5+\dl}([t_0, t_0 + T];  L^{10}(\R^2)).
\end{align*}

\end{corollary}

\begin{remark}\label{Remark3.6}\rm
(i) In terms of the current approach based on the first order expansion
\eqref{exp1}, the threshold $s = -\frac 15$ seems to be sharp.
Since we need to measure the quintic power in $L^1$ in time, 
this forces us to measure the random linear solution  essentially in $L^5$
in time.  In view of Proposition \ref{PROP:PS},  
local-in-time integrability of $t^s$ in $L^5$ requires $s > -\frac 15$.
It is worthwhile to note that  the regularity restriction $s > -\frac 15$ comes only from the temporal integrability
and does not have anything to do with the spatial integrability.

With a $p$th power nonlinearity $|u|^{p-1}u$, $p > 1$,  (in place of the quintic power $u^5$), 
a similar argument shows almost sure local well-posedness of \eqref{NLW1} for $s > -\frac 1p$, 
which is essentially sharp
(in terms of the first order expansion). 
For $p \notin 2\N + 1$, the nonlinearity is not algebraic and thus we need to proceed 
as in \cite{OOP}, where probabilistic well-posedness
of the nonlinear Schr\"odinger equations with non-algebraic nonlinearities
was studied.
See~\cite{Liu} for details.
See also Remark \ref{REM:2}.

\smallskip

\noi
(ii) It would be of interest to 
investigate if  higher order expansions, 
such as those in \cite{BOP3, OPTz}, 
give any improvement over Theorem \ref{THM:LWP2}
on almost sure local well-posedness.
One may also adapt the paracontrolled approach 
used for the stochastic NLW
\cite{GKO2, OOTol, Bring, OOTol2}
to study vNLW with random initial data.
\end{remark}

\section{Global well-posedness}\label{SEC:GWP}

In this section, we prove almost sure global well-posedness of \eqref{NLW1}.
As noted in \cite{CO, BOP2}, 
it suffices to prove the following 
 ``almost'' almost sure global well-posedness result.

\begin{proposition}\label{PROP:aasGWP}
Let $s> -\frac 15$.
Given $(u_0, u_1)  \in \H^s(\R^2)$, let $(u_0^\o, u_1^\o)$ be its Wiener randomization defined in \eqref{R1},
satisfying \eqref{B3}.
Then, given any $T, \eps > 0$, there exists a set $ \O_{T, \eps}\subset \O$
such that
\begin{itemize}
\item[\textup{(i)}]
$P( \O_{T, \eps}^c) < \eps$,

\item[\textup{(ii)}]
For each $\o \in \O_{T, \eps}$, there exists a \textup{(}unique\textup{)} solution $u$
to \eqref{NLW1}  on $[0, T]$
with $(u, \dt u)|_{t = 0} = (u_0^\o, u_1^\o)$.

\end{itemize}

\end{proposition}

It is easy to see from the Borel-Cantelli lemma that
almost sure global well-posedness (Theorem \ref{THM:GWP1}) follows
once we prove 
``almost'' almost sure global well-posedness
stated in Proposition \ref{PROP:aasGWP} above.
See~\cite{CO, BOP2}.
Hence, the remaining part of this 
section is devoted to the proof of Proposition \ref{PROP:aasGWP}.

Fix $T \gg1 $. 
In order to extend our local-in-time result for the initial value problem \eqref{NLW5} to a result on $[0, T]$ for arbitrary $T > 0$, we consider \eqref{NLW6a} with $f = z = V(t)(u_{0}^{\omega}, u_{1}^{\omega})$, given explicitly by 
\begin{align}
\begin{cases}
\dt^2 v  -  \Dl  v  + D \dt v  + (v + z)^5 =0\\
(v, \dt v) |_{t = t_0} = (v_0, v_1),
\end{cases}
\label{NLWglobal}
\end{align}
where $t_{0} \in \mathbb{R}^{+}$. In view of Corollary \ref{COR:LWP4}
and almost sure boundedness of 
the $L^{5+\dl}([0, T];  L^{10}_x(\R^2))$-norm
of the random linear solution $z (t)= V(t) (u_0^\omega, u_1^\omega)$ 
thanks  the probabilistic Strichartz estimate (Proposition \ref{PROP:PS}), 
it suffices to control the $\H^{s_0}$-norm 
of the remainder term $\vec v = (v, \dt v)$, 
where $s_0 = \frac 35 +\dl_0$ as in \eqref{lin2a} and the remainder term $v$ satisfies the initial value problem \eqref{NLW5}. This will allow us to extend the local-in-time result in Theorem \ref{THM:LWP2} to $[0, T]$ by iteratively applying Corollary \ref{COR:LWP4}.
In the next subsection, we first show a gain of regularity 
such that $(v(t), \dt v(t))$ indeed belongs to $\H^1(\R^2)$
as soon as $t > 0$.
Then, the problem is reduced to controlling the growth of the energy
$E(\vec v)$ in~\eqref{E0}, 
associated with the standard nonlinear wave equation,
since the energy $E(\vec v)$ controls  the $\mathcal{H}^{1}$-norm 
(and hence $\mathcal{H}^{s_{0}}$-norm) of the remainder term $(v, \partial_{t}v)$, as needed to establish the result. 
See Subsection \ref{SUBSEC:GWP2}.

\subsection{Gain of regularity}\label{SUBSEC:GWP1}
Consider the initial value problem \eqref{NLWglobal}. 
Fix $s > -\frac 15$ and $s_0 = \frac 35+\dl_0$ 
with small $\dl_0 > 0$ as in (the proof of) Theorem \ref{THM:LWP2}.
Let $(u_0^\o, u_1^\o)$
be the Wiener randomization of 
a given deterministic pair $(u_0, u_1)\in \H^s(\R^2)$
and fix $(v_0, v_1) \in \H^{s_0}(\R^2)$.

Let $T \gg 1$
and 
let $z(t) = V(t) (u_0^\omega, u_1^\omega)$ be 
 the random linear solution.
 Then,  it follows from 
 the probabilistic Strichartz estimate (Proposition \ref{PROP:PS})
 (see also \eqref{PS1b}) that there exists an almost surely finite
 random constant $C_\o = C_\o (T)> 0$ such that 
\begin{align}
\| z\|_{L^{5+\dl}([0,  T]; L^{10}_x)} \leq C_\o.
\label{X1}
\end{align}

\noi
Fix a good $\o\in \O$ such that $C_\o$ in \eqref{X1} is finite.
Then, from Corollary \ref{COR:LWP4}, 
we see that 
there exist $\tau_\o > 0$ and  a unique solution $\vec v = (v, \dt v)$ to \eqref{NLWglobal}
on the time interval $[t_0, t_0 + \tau_\o]$
with $(v, \dt v) |_{t = t_0} = (v_0, v_1)$
in the class
\begin{align*}
 (v, \dt v) \in C([t_0,  t_0 + \tau_\o]; \H^{s_0} (\R^2))
\qquad \text{and}\qquad 
v\in  L^{5+\dl}([t_0, t_0 + \tau_\o];  L^{10}(\R^2)).
\end{align*}

We  show that the solution $\vec v = (v, \dt v)$ to~\eqref{NLWglobal} in fact belongs
to $C((t_0,  t_0 + \tau_\o]; \H^1(\R^2))$, thanks to the smoothing due to the Poisson kernel $P(t)$ in \eqref{A2a}.
Fix $t > t_0$. By 
\eqref{A6} and Lemma \ref{LEM:Sch},
we have
\begin{align}
\|V(t-t_0) (v_0, v_1) \|_{\H^1}
\les (t-t_0)^{-1+s_0} \|(v_0, v_1 )\|_{\H^{s_0}}.
\label{G1}
\end{align}

\noi
Then, from \eqref{G1}, 
Lemma \ref{LEM:lin2} with $s = 1$, 
and \eqref{X1}, 
we have, for any $t_0 < t \leq t_0 +\tau_\o$,  
\begin{align*}
\|\vec v(t)\|_{\H^1} 
& \les   (t-t_0)^{-1+s_0} \|(v_0, v_1 )\|_{\H^{s_0}}
+ \|(v+z)^5\|_{L^1([t_0, t_0+\tau_\o]; L^2_x)}\\ 
& \les (t-t_0)^{-1+s_0}\| (v_0, v_1) \|_{\H^{s_0}}
+ \tau_\o^\ta \Big(\|v\|_{L^{5+\dl}([t_0, t_0+\tau_\o]; L^{10}_x)}
+ C_\o\Big)\\
& < \infty.
\end{align*}

\noi
This proves the gain of regularity for $\vec v = (v, \dt v)$.\footnote{Here, 
we did not show the continuity in time of $\vec v$ in $\H^1(\R^2)$
but this can be done by a standard argument, which we omit.}
In the following, our main goal is to control the $\H^1$-norm of $\vec v(t)$ on $[0, T]$
for any given $T \gg 1$.

\subsection{Energy bound}
\label{SUBSEC:GWP2}

\medskip
Fix $\eps > 0$.
Then, it follows from Theorem \ref{THM:LWP2} 
that there exists $\O_{T_0}$ with sufficiently small $T_0 = T_0(\eps) > 0$ such that 
\begin{align}
P( \O_{T_0}^c) < \frac \eps 2
\label{G2}
\end{align}

\noi
and, for each $\o \in \O_{T_0}$, 
 the local well-posedness of \eqref{NLW1} holds on $[0, T_0]$.

Fix a large target time $T \gg 1$.
In the following, by excluding further a set of small probability, 
we construct 
the solution $\vec v = (v, \dt v)$
on the time interval $[T_0, T]$ and hence on $[0, T]$.
Our
goal is to control the growth of the $\H^1$-norm of $\vec v(t)$ on $[T_0, T]$. To do this, we closely follow the procedure for a similar energy bound for a defocusing quintic nonlinear wave equation in \cite{OP} in the remainder of this section. Since the same argument in the proof of Proposition 4.1 in \cite{OP} applies to establishing the energy bound in the current context, we only summarize the main steps below and refer the reader to \cite{OP} for details.

\medskip

\textbf{Step 1:} \textit{Reduction to an energy bound.} In order to control the $\mathcal{H}^{1}$ norm of $\vec v(t)$ on $[T_{0}, T]$, it suffices to control the $\dot{\mathcal{H}}^{1}$ norm on $[T_{0}, T]$, where
$\dot \H^1(\R^2)
:= \dot H^1(\R^2) \times L^2(\R^2)$. This is due to the fundamental theorem of calculus:
\begin{align*}
\|v (t)\|_{L^2_x}
&=\bigg\|\int_0^t\dt v(t')dt'\bigg\|_{L^2_x} 
\leq T\|\pa_t v\|_{L^\infty_T L^2_x},
\end{align*}

\noi
for $ 0 < t \leq T$. The $\dot{\mathcal{H}}^{1}$ norm of $\vec v$ is further controlled by the energy $E(\vec v)$ in \eqref{E0}. \textit{Hence, it suffices to control the energy $E(\vec v)$ on $[T_{0}, T]$.}

\medskip

\textbf{Step 2:} \textit{Statement of desired energy growth inequality on $[T_{0}, T]$.} We will derive an energy inequality to estimate $E(\vec v)(t) - E(\vec v)(T_{0})$ for $t \in [T_{0}, T]$. We remark that if we were considering the case of the cubic nonlinearity (rather than a quintic nonlinearity), we can follow
the Gronwall argument by Burq and Tzvetkov \cite{BT3}.
In the current quintic case, however, this argument fails.
To overcome this difficulty, 
we employ the integration-by-parts trick
introduced by Pocovnicu and the second author~\cite{OP}
in studying almost sure global well-posedness
of the energy-critical defocusing quintic NLW on $\R^3$.


Let $z(t) = V(t) (u_0^\o, u_1^\o)$ be the random linear solution defined in  \eqref{z1}.
With $\wt V(t)$ as in~\eqref{z3}, define $\wt z$ by 
\begin{align}
\wt z (t) = \jb{\nb}^{-1} \dt z(t) 
= \wt V(t) (u_0^\o, u_1^\o).
\label{z2}
\end{align}

\noi
Then, 
given $0 < T_0 < T$, we set $A(T_0, T)$ as
\begin{align}
\begin{split}
A(T_0, T) 
& = 1 + \|z \|^2_{L^\infty([T_0, T];L^\infty_x)} 
+ 
 \|z \|^{10}_{L^{10}([T_0, T]; L^{10}_x )}
 + \|z\|_{L^\infty([T_0, T];  L^6_x)}^6 \\
& \quad  + \|\wt z \|_{L^6([T_0, T]; L^6_x)}^6
+ \big\|\jb{\nb}^{s_1}  \wt z\big\|_{L^\infty([T_0, T]; L^\infty_x)}, 
\end{split}
\label{E1a}
\end{align}

\noi
where $s_1 > \frac 12$ is sufficiently close to $\frac 12$ 
(to be chosen later). Since we can control $A(T_{0}, T)$ with high probability by the probabilistic Strichartz estimate in Proposition \ref{PROP:PS2}, our goal is to obtain an energy inequality of the following form and apply Gronwall's inequality:
\begin{align}\label{energycontrol}
 E(\vec v)(t)
\les E(\vec v)(T_0) + A(T_0, T) + A(T_0, T) \int_{T_0}^t E(\vec v)(t') dt'.
\end{align}

\textbf{Step 3:} \textit{Calculation of $\frac{d}{dt}E(\vec v)$.} By using the equation \eqref{NLW5}, we have
\begin{equation}
\frac{d}{dt} E(\vec v)(t) =- \int_{\R^2}(D^\frac{1}{2}\dt v)^2  dx
- \int_{\R^2}\dt v  \big((v+z)^5  - v^5\big) dx.
\label{E1b}
\end{equation}
By using the fact that $- \int_{\R^2}(D^\frac{1}{2}\dt v)^2  dx \le 0$ and proceeding as in \cite{OP}, we deduce that
\begin{align}
\begin{split}
E( \vec v)(t) - 
&  E(\vec v)(T_0)
 \leq 
- 
\int_{T_0}^t
\int_{\R^2}  z (t') \dt ( v(t') ^5) dt' dx
- \int_{T_0}^t\int_{\R^2}
\dt v(t') \mathcal N(z, v)(t')  dx dt' \\
& =:\1(t) +\II(t)
\end{split}
 \label{E1}
\end{align}
for any $t \in [T_0, T]$, where $\NN(z, v)$ denotes the lower order terms in $v$:
\[\NN(z, v) = 10 z^2 v^3 + 10 z^3 v^2 + 5 z^4 v + z^5.\]

\textbf{Step 4:} \textit{Estimate of $\II(t)$.} We have reduced the goal of showing the energy bound estimate \eqref{energycontrol} to estimating the quantities $\1(t)$ and $\II(t)$. Using the same arguments in \cite{OP}, we have that 
\begin{align}
\begin{split}
|\II(t)| & \les \big(1+ \|z \|^2_{L^\infty([T_0, T];L^\infty_x)} \big)
\int_{T_0}^t E(\vec v)(t') dt' 
+ \|z \|^{10}_{L^{10}([T_0, T]; L^{10}_x )}\\
& \le A(T_0, T)
\int_{T_0}^t E(\vec v)(t') dt' 
+ A(T_0, T).
\end{split}
\label{E2}
\end{align}

\textbf{Step 5:} \textit{Estimate of $\1(t)$.} The estimate of $\1(t)$ is more involved. We proceed by using the integration-by-parts trick in \cite{OP} and integrating by parts in time, to obtain two quantities $\1_1(t)$ and $\1_2(t)$:
\begin{align}
\1(t)  = 
 - \int_{\R^2}  z (t')  v(t') ^5 dx\bigg|_{T_0}^t
+ \int_{\R^2} \int_{T_0}^t \dt z (t')  v(t') ^5dt' dx
=:\1_1(t')\Big|_{T_0}^t +\1_2(t).
\label{E3}
\end{align}

\noi
As for the first term $\1_1$, we use Young's inequality with exponents $6$ and $6/5$ to obtain 
\begin{align}
\begin{split}
|\1_1 (t) - \1_1(T_0)|
& \les \eps_0^{-6}\|z (T_0)\|_{L^6_x}^6  + \eps_0^\frac{6}{5} \|v(T_0)\|^6_{L^6_x} + \eps_0^{-6}\|z (t)\|_{L^6_x}^6  + \eps_0^\frac{6}{5} \|v(t)\|^6_{L^6_x}\\
& \les \eps_0^{-6} \|z\|_{L^\infty([T_0, T];  L^6_x)}^6 
+ \eps_0^\frac{6}{5} E(\vec v)(T_0)
 + \eps_0^\frac{6}{5} E(\vec v)(t)
\end{split}
\label{E4}
\end{align}

\noi
for some small constant $\eps_0 > 0$ (to be chosen later).

Next, we consider  the  second term $\1_2$ in \eqref{E3}.
While we closely follow the argument in~\cite{OP}, 
we summarize the procedure here for readers' convenience.
From \eqref{z2}, we have 
\begin{equation}
\1_2 (t) =  \int_{T_0}^t 
\int_{\R^2} \jb{\nabla} \wt z (t') \cdot  v(t') ^5 dx dt'.
\label{E4a}
\end{equation}

\noi
Given dyadic $M \geq 1$, let $\Q_M$
denote the \textit{nonhomogeneous} Littlewood-Paley projector onto the (spatial)
frequencies $\{|\xi|\sim M\}$. This means that $\Q_1$ is a smooth  projector onto the (spatial)
frequencies $\{|\xi|\les 1\}$ and by convention, $\Q_{2^{-1}} = 0$.
Then, 
define $\mathcal{I}(t)$ by 
\begin{equation*}
\I (t) : \!  =  \int_{\R^2} \jb{\nabla} \wt z (t)  \cdot v(t) ^5 dx
\end{equation*} 
and note that by using a Littlewood-Paley frequency decomposition, we have that
\begin{equation}\label{dyadicI2}
\mathcal{I}(t) \sim \sum_{\substack{M \geq 1\\ \text{dyadic}}}\I^M(t)
\end{equation}
for dyadic $M = 2^{k}$ with $k$ a nonnegative integer, and 
\begin{equation*}
\mathcal{I}^{M}(t) := \sum_{k = -1}^1
M \int_{\R^2} \Q_{2^k M} \wt z (t)  \Q_M\big( v(t) ^5\big) dx.
\end{equation*}

\noi
We also set 
\[\I^{M\geq 2}(t) = \I(t) - \I^1(t).\]

\medskip

\noi
$\bullet$ {\bf Case 1:} $M = 1$.
\quad 
In this case, we can bound the contribution 
to $|\1_2(t)|$ 
by Young's inequality as
\begin{align}
\begin{split}
\bigg|\int_{T_0}^t \I^1(t') dt'\bigg|
& \les \|\wt z \|_{L^6([T_0, T]; L^6_x)}^6
+ \int_{T_0}^t \|v (t')  \|_{L^6_{x}}^6 dt'\\
& \les A(T_0, T)
+ \int_{T_0}^t E(\vec v)(t') dt'.
\end{split}
\label{E4b}
\end{align}

\smallskip

\noi
$\bullet$ {\bf Case 2:} $M \geq 2$.
\quad 
We can follow the full details given in \cite{OP} and apply the Littlewood-Paley decomposition on each factor of $v$ in the quintic term $v^{5}$, in order to obtain the following estimate:
\begin{equation}
 |\I^{M\ge 2} (t)| \les 
\big\|\jb{\nb}^{s_2-\ta} \wt  z(t)\big\|_{L^\infty_x}E(\vec v),
\label{E4c}
\end{equation}

\noi
provided that $2(1-s_2+\ta+) \leq 1$.
Hence, 
by setting  $s_1 = s_2 -\ta$, it follows 
from \eqref{E1a} and \eqref{E4c} that 
\begin{align}
 \bigg|\int_{T_0}^t \I^{M\ge 2} (t')dt'\bigg|
& \les 
A(T_0, T) \int_{T_0}^t E(\vec v)(t') dt',
\label{E4d}
\end{align}

\noi
provided that $2(1-s_2+\ta+) \leq 1$, namely 
$s_2 \ge \frac 12 + \ta+$, 
which is satisfied by choosing $s_2 > \frac 12$ and $\ta > 0$ sufficiently small.
 This determines the choice of $s_1 = s_2 - \ta$ in \eqref{E1a}.
 
Therefore, from 
 \eqref{E4a}, \eqref{dyadicI2}, \eqref{E4b}, and \eqref{E4d},  we obtain
\begin{align}
\begin{split}
|\1_2 (t)| 
& \les 
\|\wt z \|_{L^6([T_0, T]; L^6_x)}^6
+ 
\Big(1+\big\|\jb{\nb}^{s_1}  \wt z\big\|_{L^\infty([T_0, T]; L^\infty_x)}\Big)
\int_{T_0}^t 
E(\vec v)(t') dt'\\
& \les
A(T_0, T)
+ 
A(T_0, T)
\int_{T_0}^t 
E(\vec v)(t') dt'.
\end{split}
\label{E5}
\end{align}

\textbf{Step 6:} \textit{Final estimate and Gronwall inequality.}  Putting
\eqref{E1}, \eqref{E2}, \eqref{E3}, \eqref{E4}, and \eqref{E5} together
and choosing sufficiently small $\eps_0 > 0$ in \eqref{E4},  we obtain
\begin{align*}
 E(\vec v)(t)
\les E(\vec v)(T_0) + A(T_0, T) + A(T_0, T) \int_{T_0}^t E(\vec v)(t') dt',
\end{align*}

\noi
for any $t \in [T_0, T]$.
Therefore, from Gronwall's inequality,
we conclude that 
\begin{align}
 E(\vec v)(t)
\les C\big(T_0, T, E(\vec v)(T_0), A(T_0, T)\big)
\label{Ex}
\end{align}

\noi
for any $t \in [T_0, T]$.

\begin{remark}\label{REM:2}\rm
(i) In order to justify 
the formal computation in this subsection, 
we need to proceed with the smooth solution $(v_N , \dt v_N)$
associated with the frequency truncated random initial data
(for example, to guarantee  finiteness
of the term $- \int_{\R^2}(D^\frac{1}{2}\dt v)^2  dx$ in~\eqref{E1b})
and then take $N \to \infty$, using the approximation argument (Corollary \ref{COR:LWP3}).
This argument, however, is standard and thus we omit details.
See, for example,  \cite{OP}.

\smallskip

\noi
(ii) In this section, we followed the argument in \cite{OP}
to obtain an  energy bound in the quintic case.
 In this argument, 
the first term after the first inequality in \eqref{E2} 
provides the restriction $p \leq 5$ 
on the degree of the nonlinearity $|u|^{p-1} u $.
For $p > 5$, we will need to apply 
the integration by parts trick to lower order terms
as well.  See for example \cite{Latocca}
in the context of the standard NLW.
In a recent preprint \cite{Liu}, 
Liu extended Theorems \ref{THM:LWP1} and \ref{THM:GWP1}
to the super-quintic case ($p > 5$)
and proved almost sure global well-posedness
of the defocusing vNLW \eqref{vNLW1} 
in $\H^s(\R^2)$ for $s > -\frac 1p$.

%
%
\end{remark}

\subsection{Proof of Proposition \ref{PROP:aasGWP}}

Fix a target time $T \gg 1$ and small $\eps > 0$.
Then, let $T_0$ be as in 
\eqref{G2}.
With $A(T_0, T)$ as in \eqref{E1a}, 
set
\[A_\ld = \big\{ \o \in \O : A(T_0, T) < \ld\big\}\]

\noi
for $\ld > 0$.
From Proposition \ref{PROP:PS2}, 
there exists  $\ld_0 \gg 1$ such that 
\begin{align}
P(A_{\ld_0}^c) < \frac \eps 2.
\label{G4}
\end{align}
	
\noi
Now, set $\O_{T, \eps} = \O_{T_0} \cap A_{\ld_0}$.
Then, from \eqref{G2} and \eqref{G4}, 
we have $P(\O_{T, \eps}^c) < \eps$.

Let  $\o \in \O_{T, \eps}$.
From 
\eqref{E1a}
and H\"older's inequality, we see that $A(T_0, T)$ controls
the $L^{5+\dl}_t([T_0, T]; L^{10}_x )$-norm of $z$:
\[  \|z \|_{L^{5+\dl}([T_0, T]; L^{10}_x )}
\les T^\ta \ld_0^\frac{1}{10}\]

\noi
for some $\ta > 0$, 
where $\dl > 0$ is as in (the proof of) Theorem \ref{THM:LWP2}.
Then, together with  the energy bound \eqref{Ex} and the discussion in Section \ref{SUBSEC:GWP1}, 
we can iteratively apply 
Corollary \ref{COR:LWP4}
(see also the discussion right after Proposition~\ref{PROP:aasGWP})
and construct 
a  solution $u = z + v$
to \eqref{NLW1}  on $[0, T]$
with $(u, \dt u)|_{t = 0} = (u_0^\o, u_1^\o)$
for each   $\o \in \O_{T, \eps}$.
This proves Proposition  \ref{PROP:aasGWP}
and hence almost sure global well-posedness
(Theorem \ref{THM:GWP1}).

\begin{ackno}\rm

This work was partially supported by the National Science Foundation under grants DMS-1853340, and DMS-2011319 (\v{C}ani\'{c} and Kuan), and 
by the European Research Council under grant number 864138 ``SingStochDispDyn'' (Oh).

\end{ackno}


\end{document}